\documentclass[12pt,onecolumn,journal,letterpaper]{IEEEtran}%

\usepackage{soul}
\usepackage[normalem]{ulem}
\usepackage{microtype}

\usepackage[hidelinks]{hyperref}
\usepackage{bookmark}
\usepackage[latin1]{inputenc}
\usepackage[english]{babel}

\usepackage[dvipsnames]{xcolor}
\usepackage{amsthm,amsfonts,amssymb,amsmath,color}
\usepackage{graphicx}
\usepackage{placeins}
\usepackage{algorithm}
\usepackage[]{algpseudocode}

\usepackage{mathtools}

\usepackage{tikz}
\usetikzlibrary{shapes,arrows,calc}

\newcommand{\map}[3]{#1: #2 \rightarrow #3}

\newcommand{\blkdiag}{\operatorname{diag}}

\newcommand{\EE}{\mathcal{E}} \newcommand{\FFF}{\mathcal{F}}
\newcommand{\GG}{\mathcal{G}}

\newcommand{\OO}{\mathcal{O}}

\DeclareMathOperator{\rank}{rank}

\newcommand{\subj}{\textnormal{subj. to}}

\newcommand{\nbrs}{\mathcal{N}}
\newcommand{\subgrad}{\widetilde{\nabla}}

\newcommand{\real}{{\mathbb{R}}}
\renewcommand{\natural}{{\mathbb{N}}}

\newcommand{\expv}{{\mathbb{E}}}

\newcommand{\until}[1]{\{1,\ldots,#1\}}

\newcommand\oprocendsymbol{\hbox{$\square$}}
\newcommand\oprocend{\relax\ifmmode\else\unskip\hfill\fi\oprocendsymbol}
\def\eqoprocend{\tag*{$\square$}}

\newtheorem{theorem}{Theorem}[section]
\newtheorem{proposition}[theorem]{Proposition}
\newtheorem{corollary}[theorem]{Corollary}
 \newtheorem{lemma}[theorem]{Lemma}
\newtheorem{remark}[theorem]{Remark}

\newtheorem{assumption}[theorem]{Assumption}

\newcommand{\bx}{\mathbf{x}}
\newcommand{\br}{\mathbf{r}}

\newcommand{\inc}{\Gamma}

\newcommand{\bz}{\mathbf{z}}

\newcommand{\bv}{\mathbf{v}}

\newcommand{\by}{\mathbf{y}}

\newcommand{\bg}{\mathbf{g}}

\newcommand{\0}{\mathbf{0}}
\newcommand{\1}{\mathbf{1}}

\newcommand{\bmu}{\boldsymbol{\mu}}

\newcommand{\blambda}{\boldsymbol{\lambda}}

\newcommand{\kron}{\otimes}

\newcommand{\tp}{\tilde{p}}
\newcommand{\prob}{\sigma}
\newcommand{\rvedge}{\nu}
\newcommand{\cvmat}{\Pi}
\newcommand{\best}{\mathrm{best}}

\let\Im\relax
\DeclareMathOperator{\Im}{Im}
\DeclareMathOperator{\Ker}{Ker}

\makeatletter
\newcommand{\StatexIndent}[1][3]{%
  \setlength\@tempdima{\algorithmicindent}%
  \Statex\hskip\dimexpr#1\@tempdima\relax}
\makeatother

\renewcommand{\inf}{\operatornamewithlimits{inf\vphantom{p}}}

\renewcommand{\lim}{\operatornamewithlimits{lim\vphantom{p}}}

\newcommand{\GN}[1]{{ #1}}

\graphicspath{{figs/}}

\def \algname/{Distributed Primal Decomposition for Time-Varying graphs}
\def \algacronym/{DPD-TV}

\title{\Large \bf
Distributed Constraint-Coupled Optimization via \\ Primal Decomposition over Random Time-Varying Graphs
}

\author{Andrea Camisa,
  Francesco Farina,
  Ivano Notarnicola,
  Giuseppe Notarstefano
  \thanks{The authors are with the Department of Electrical, 
  Electronic and Information Engineering, Alma Mater Studiorum -- Universit\`{a} di Bologna, Bologna, Italy,
  \texttt{\{a.camisa, franc.farina, ivano.notarnicola, giuseppe.notarstefano\}@unibo.it}.
  This result is part of a project that has received funding from the European Research
  Council (ERC) under the European Union's Horizon 2020 research and innovation
  programme (grant agreement No 638992 - OPT4SMART).
  }
  \thanks{A preliminary version of this work has appeared in
  the Proceedings of the 58th Conference on Decision and Control~\cite{camisa2019primal}.
  The present manuscript provides all the theoretical proofs under a more general
  communication model, with nonuniform edge probabilities.
  Moreover, convergence rates are established and a discussion about the algorithm
  tuning is provided.}
}

\begin{document}

\maketitle

\begin{abstract}
  The paper addresses large-scale, convex optimization problems that need to be
  solved in a distributed way by agents 
  communicating according to a random
  time-varying graph.
  Specifically, the goal of the network
  is to minimize the sum of local costs,
  while satisfying local and coupling constraints.
  Agents communicate according to a time-varying model in which edges of an
  underlying connected graph are active at each iteration with certain
  non-uniform probabilities.
  By relying on a primal decomposition scheme applied to an equivalent problem
  reformulation, we propose a novel distributed algorithm in which agents
  negotiate a local allocation of the total resource only with neighbors with
  active communication links.
  The algorithm is studied as a subgradient method with block-wise updates, in
  which blocks correspond to the graph edges that are active at each iteration.
  Thanks to this analysis approach, we show almost sure convergence to the
  optimal cost of the original problem and almost sure asymptotic primal
  recovery without resorting to averaging mechanisms typically employed in dual
  decomposition schemes.
  Explicit sublinear convergence rates are provided under the assumption of
  diminishing and constant step-sizes.
  Finally, an extensive numerical study on a plug-in electric vehicle charging
  problem corroborates the theoretical results.
\end{abstract}

\section{Introduction}
\label{sec:intro}

Large-scale systems consisting of several independent control systems can be
found in numerous contexts ranging from smart grids to autonomous vehicles and
cooperative robotics.
In order to perform cooperative control tasks, such systems (or agents)
must employ their computation capabilities and collaborate with each other
by means of neighboring communication, without resorting to a centralized
computing unit.
These cooperative tasks can be often formulated as distributed optimization
problems consisting of a large number of decision variables,
each one associated to an agent in the network and satisfying private constraints.
Furthermore, a challenging feature of such optimization problems is that
all the decision variables are intertwined by means of a global coupling constraint,
that can be used to model, e.g., formation maintenance requirements or
a total budget that must not be exceeded.
This set-up is referred to as \emph{constraint coupled} optimization.

The majority of the literature on distributed optimization has focused on a
framework in which, differently from the constraint-coupled set-up, cost functions and
constraints depend on the same, common decision variable, and agents
aim for \emph{consensual} optimal solutions.
An exemplary, non-exhaustive list of works for this optimization set-up is
\cite{nedic2009distributed,duchi2012dual,zhu2012distributed,
mota2013dadmm,shi2014linear,jakovetic2014fast,shi2015extra}.
Only recently has the constraint-coupled set-up gathered more attention
from our community, due to its applicability in control.
In~\cite{simonetto2016primal} consensus-based dual decomposition is combined
with a primal recovery mechanism, whereas~\cite{falsone2017dual} considers a
distributed dual algorithm based on proximal minimization.
In~\cite{notarnicola2017constraint} a distributed algorithm based on successive
duality steps is proposed. Differently from~\cite{simonetto2016primal,falsone2017dual},
which employ running averages for primal recovery, \cite{notarnicola2017constraint}
can guarantee feasibility of primal iterates without averaging schemes.
In~\cite{chang2014distributed} a
consensus-based primal-dual perturbation algorithm is proposed to solve smooth
constraint-coupled optimization problems. A distributed saddle-point
algorithm with Laplacian averaging is proposed in~\cite{mateos2017distributed}
for a class of min-max problems. 
In~\cite{burger2014polyhedral}, a distributed algorithm based on cutting planes is formulated.
Recently, in~\cite{liang2019distributed} a primal-dual algorithm
with constant step-size is proposed under smoothness assumption of both 
costs and constraints.
The works in~\cite{necoara2015linear, alghunaim2018dual, sherson2019distributed}
consider a similar set-up, but the proposed algorithms strongly rely on the sparsity 
pattern of the coupling constraints.
Linear constraint-coupled problem set-ups have been also tackled by means of distributed 
algorithms based on the Alternating Direction Method of Multipliers (ADMM). In~\cite{chang2014multi}
the so-called consensus-ADMM is applied to the dual problem formulation, which is
then tailored for an application in Model Predictive Control by~\cite{wang2017distributed}.
In~\cite{carli2019distributed} an ADMM-based algorithm is proposed and
analyzed using an operator theory approach while in~\cite{zhang2018consensus}
an augmented Lagrangian approach equipped with a tracking mechanism is proposed.
However, the last two approaches require agents to perform multiple communication rounds 
to converge in a neighborhood of an optimal solution. In~\cite{falsone2019tracking}
ADMM is combined with a tracking mechanism to design a distributed algorithm
with exact convergence to an optimal solution.

\GN{
The analysis of our algorithm for random time-varying graphs} builds on
randomized block subgradient methods, \GN{therefore let us} recall some related
works from the centralized literature.
A survey on block coordinate methods is given in~\cite{beck2013convergence},
while a unified framework for nonsmooth problems can be found~\cite{razaviyayn2013unified}.
In~\cite{richtarik2014iteration}, a randomized block coordinate descent method
is formulated, whereas~\cite{dang2015stochastic} investigates a stochastic block mirror
descent approach with random block updates.
\GN{In~\cite{necoara2013random}, a distributed algorithm for a linearly constrained
problem is analyzed with coordinate descent methods.
This technique is also used in~\cite{necoara2014random}, which considers a constraint-coupled
problem. However, the approach used in~\cite{necoara2013random,necoara2014random}
only allow for a single pair of agents updating at a time and requires smooth
cost functions.}

In this paper, we propose a distributed algorithm to solve \GN{nonsmooth}
constraint-coupled optimization problems over random, time-varying communication networks. 
We consider a communication model in which edges of an underlying, connected graph 
have a certain probability of being active at each time step.
The proposed algorithm consists in a two-step procedure in which agents first solve
a local optimization problem and then update a vector representing the local
allocation of total resource. 
\GN{
The algorithmic structure is inspired to the algorithm for fixed graphs
in~\cite{notarnicola2017constraint}. However, the line of analysis
proposed in~\cite{notarnicola2017constraint} hampers extension to
time-varying graphs. Therefore, in this paper, we develop a new
theoretical analysis to deal with the significant challenges arising in
the time-varying context. In particular,} this method is interpreted as a primal decomposition scheme
applied to an equivalent, relaxed version of the target constraint-coupled
problem.
For this scheme, we prove that almost surely the objective value converges to
the optimal cost, and any limit point of the local solution estimates is an
optimal (feasible) solution.
Moreover, we prove a sublinear convergence rate of the objective value under the
assumption of constant or diminishing step-size.  As for constant step-size,
convergence to a neighborhood of the solution is attained with a rate
$O(1/\sqrt{t})$, while for a diminishing step-size of the type $1/t$,
exact convergence is attained with rate $O(1/\log(t))$.
To show these results, we employ a graph-induced change of variables to derive
an equivalent, unconstrained problem formulation. This allows us to recast the
distributed algorithm as a randomized block subgradient method in which blocks
correspond to edges in the graph.
As a side result, we also provide an almost sure convergence result for a block
subgradient method in which (multiple) blocks are drawn according to non-uniform
probabilities. This generalized block subgradient method results into updates in
which different combinations of multiple blocks can be chosen. To the best of
our knowledge, these nontrivial challenges have not been addressed so far in the
\GN{block subgradient} literature.
\GN{A thorough comparison of the contributions provided in this paper
with existing work will be performed in light of the analysis
provided in Section~\ref{sec:analysis}.}

The paper is organized as follows. In Section~\ref{sec:set-up_and_algorithm}, we
introduce the distributed optimization set-up and we describe the proposed
distributed algorithm. In Section~\ref{sec:block_subgrad_method}, we provide
intermediate results on a (centralized) block subgradient method, which are then
used in Section~\ref{sec:analysis} for the analysis of the distributed
algorithm. Convergence rates and a discussion on algorithm tuning are
enclosed in Section~\ref{sec:rates_and_discussion}. Finally, in
Section~\ref{sec:simulations}, an extensive numerical study on a control
application is presented.

\paragraph*{Notation} %
The symbols $\0$ and $\1$ denote the vector of zeros
and ones respectively. %
The $n \times n$ identity matrix is denoted by $I_n$. Where
the size of the matrix is clear from the context, we drop the subscript $n$.
Given a vector $\bx \in \real^n$ and a positive definite matrix
$W \in \real^{n \times n}$, we denote by $\|\bx\|_W = \sqrt{\bx^\top W \bx}$
the norm of $\bx$ weighted by $W$, which we also term $W$-norm.
Given two vectors $\bx,\by\in\real^n$ we write $\bx \le \by$ (and
consistently for other sides) to denote component-wise inequalities.
The symbol $\kron$ denotes the Kronecker product.
Given a convex function $f(\bx) : \real^n \rightarrow \real$
and a vector $\bar{\bx} \in \real^n$, we denote by
$\subgrad f(\bar{\bx})$ a subgradient of $f$ at $\bar{\bx}$.
Given a vector $\bz$ arranged in $m$ blocks, its $\ell$-th block
(or portion) is denoted by $\bz_\ell$ or, interchangeably, by $[\bz]_\ell$,
and the complete vector is written $\bz = (\bz_1, \ldots, \bz_m)$.

\section{Optimization Set-up and Distributed Algorithm}
\label{sec:set-up_and_algorithm}

In this section, we formalize the investigated problem and
network set-up. Then, we present the proposed
distributed algorithm together with its convergence result.
Finally, we recall some preliminaries for the subsequent analysis.

\subsection{Distributed Constraint-Coupled Optimization}
\label{sec:set-up}

We deal with a network of $N$ agents that must solve
a \emph{constraint-coupled} optimization problem, which can be
stated as follows
\begin{align}
\label{eq:problem_original}
\begin{split}
  \min_{\bx_1, \ldots, \bx_N} \: & \: \sum_{i=1}^N f_i(\bx_i)
  \\
  \subj \: & \: \sum_{i=1}^N \bg_i(\bx_i) \le \0,
  \\
  & \: \bx_i \in X_i, \hspace{1cm} i \in \until{N},
\end{split}
\end{align}
where $\bx_1,\ldots,\bx_N$ are the decision variables with each $\bx_i \in \real^{n_i}$, $n_i \in \natural$.
Moreover, for all $i \in \until{N}$, $\map{f_i}{\real^{n_i}}{\real}$
depends only on $\bx_i$, $X_i \subset \real^{n_i}$ is the
constraint set associated to $\bx_i$ and
$\map{\bg_i}{\real^{n_i}}{\real^S}$ is the $i$-th contribution to the
(vector-valued) \emph{coupling} constraint $\sum_{i=1}^N \bg_i(\bx_i) \le \0$.

In the considered distributed computation framework, the problem data
are assumed to be scattered throughout the network. Agents have only a
partial knowledge of the entire problem and must cooperate with each other
in order to find a solution.
Each agent $i$ is assumed to know only its local constraint $X_i$,
its local cost $f_i$ and its own contribution $\bg_i$ to the
coupling constraints, and is only interested in computing its own
portion $\bx_i^\star$ of an optimal solution
$(\bx_1^\star, \ldots, \bx_N^\star)$ of problem~\eqref{eq:problem_original}.

The following two assumptions guarantee that
\emph{(i)} the optimal cost of problem~\eqref{eq:problem_original} is finite
and at least one optimal solution exists,
\emph{(ii)} duality arguments are applicable.
\begin{assumption}
\label{ass:problem}
  For all $i \in \until{N}$, the set $X_i$ is non-empty, convex and compact, the
  function $f_i$ is convex and each component of $\bg_i$ is a convex function.
  \oprocend
\end{assumption}
\begin{assumption}[Slater's constraint qualification]
\label{ass:slater}
  There exist $\bar{\bx}_1 \in X_1, \ldots, \bar{\bx}_N \in X_N$ such that
  $\sum_{i=1}^N \bg_i(\bar{\bx}_i) < \0$.
  \oprocend
\end{assumption}

\subsection{Random Time-Varying Communication Model}
\label{sec:communication model}

Agents are assumed to communicate according to a time-varying
communication graph, obtained as subset of an \emph{underlying} graph
$\GG_u = (\until{N}, \EE_u)$, assumed to be undirected and connected, where
$\EE_u \subseteq \until{N} \times \until{N}$ is the set of edges.
An edge $(i,j)$ belongs to $\EE_u$ if and only if agents $i$ and $j$ can
transmit information to each other, in which case also $(j,i) \in \EE_u$.
In many applications, the communication links are not always active
(due, e.g., to temporary unavailability). This is taken into account
by considering that each undirected edge $(i,j)\in\EE_u$ has a probability
$\prob_{ij} \in (0,1]$ of being active.
As a result, the actual communication network is a random, time-varying graph
$\GG^t = (\until{N}, \EE^t)$, where $t \in \natural$ represents a universal
time index and $\EE^t\subseteq\EE_u$ is the set of active edges at time $t$.
The set of neighbors of agent $i$ in $\GG^t$ is denoted by
$\nbrs_i^t = \left\{j \in \until{N} \mid (i,j) \in \EE^t \right\}$. Consistently,
the set of neighbors of agent $i$ in the underlying graph $\GG_u$
is denoted by $\nbrs_{i,u}$.

Let us define $\rvedge_{ij}^t$ as the Bernoulli random variable that is equal
to $1$ if $(i,j) \in \EE^t$ and $0$ otherwise, for all $(i,j) \in \EE_u$ with
$j > i$ and $t \ge 0$.
The following assumption is made.
\begin{assumption}
\label{ass:iid_var}
  For all $(i,j) \in \EE_u$ with $j > i$, the random variables
  $\{\rvedge_{ij}^t\}_{t \ge 0}$ are independent and identically
  distributed (i.i.d.). Moreover, for all $t \ge 0$, the random variables
  $\{\rvedge_{ij}^t\}_{(i,j) \in \EE_u, \: j > i}$ are mutually independent.
  \oprocend
\end{assumption}

\noindent A pictorial representation of the time-varying communication
model is provided in Figure~\ref{fig:network_model}.
\begin{figure}[htbp]\centering
\vspace{0.15cm}

  \includegraphics{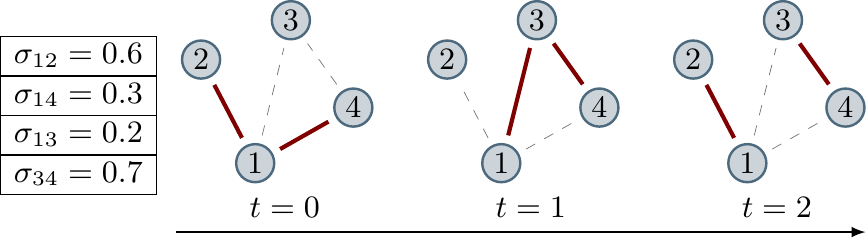}
  
  \caption{
    Example of random time-varying network with $N = 4$ agents.
    Active edges are denoted with red lines, while inactive edges are depicted
    with dashed gray lines. The (connected) underlying graph is the union of
    all such edges, and the activation probabilities are specified in the table.
  }
\label{fig:network_model}
\end{figure}

\subsection{Distributed Algorithm Description}
\label{sec:algorithm_description}

Let us now introduce the \algname/ (\algacronym/) algorithm to compute an optimal solution
$(\bx_1^\star, \ldots, \bx_N^\star)$ of problem~\eqref{eq:problem_original}.
Informally, the algorithm works as follows.
Each agent stores and updates a local solution estimate $\bx_i^t \in \real^{n_i}$
and the auxiliary variables $\rho_i^t \in \real, \bmu_i^t, \by_i^t \in \real^{S}$.
At the beginning, the variable $\by_i^t$ is initialized such that
$\sum_{i=1}^N \by_i^0 = \0$ (e.g., $\by_i^0 = \0$ for all $i$).
At each iteration $t$, agents solve a local optimization problem using the
current value of $\by_i^t$. The variables $(\bx_i^t, \rho_i^t)$ are set to the
primal solution of this problem, where $\bx_i^t$ forms an estimate of
$\bx_i^\star$ and $\rho_i^t$ is \GN{%
a transient} violation of the coupling
constraints (more details are given in Section~\ref{sec:relaxation_primal_decomp}).
The variable $\bmu_i^t$ is set to the dual solution of the problem and,
together with the information gathered from neighbors,
is used to update $\by_i^t$.

Formally, let $\alpha^t \ge 0$ denote the step-size and let $M > 0$ be a
tuning parameter (see Section~\ref{sec:discussion_M} for a discussion).
The next table summarizes the \algacronym/ algorithm from the perspective of node $i$,
where the notation ``$\bmu_i :$'' in~\eqref{eq:alg_local_prob}
means that $\bmu_i$ is the Lagrange multiplier associated to %
$\bg_i(\bx_i) \leq \by_i^t + \rho_i\1$.

\begin{algorithm}[H]
\renewcommand{\thealgorithm}{}
\floatname{algorithm}{Algorithm}

  \begin{algorithmic}[0]
    
    \Statex \textbf{Initialization}: $\by_i^0$ such that $\sum_{i=1}^N \by_i^0 = \0$
    \medskip

    \Statex %
    \textbf{For} $t = 0, 1,  2, \ldots$
    \medskip
    
      \StatexIndent[0.75]
      \textbf{Compute} $((\bx_i^t, \rho_i^t), \: \bmu_i^t)$ as a primal-dual
      solution of
      \begin{align}
      \label{eq:alg_local_prob}
      \begin{split}
        \min_{\bx_i, \rho_i} \hspace{1.2cm} &\: f_i(\bx_i) + M \rho_i
        \\
        \subj \:
        \:\: \bmu_i : \:
        & \: \bg_i(\bx_i) \leq \by_i^t + \rho_i\1
        \\
        & \: \bx_i \in X_i, \: \: \rho_i \ge 0
      \end{split}
      \end{align}

      \StatexIndent[0.75]
      \textbf{Gather} $\bmu_j^t$ from $j \in \nbrs_i^t$ and update
      \begin{align}
        \label{eq:alg_update}
        \by_i^{t+1} = \by_i^t + \alpha^t \sum_{j \in \nbrs_i^t} \big( \bmu_i^t - \bmu_j^t \big)
      \end{align}

  \end{algorithmic}
  \caption{\algacronym/}
  \label{alg:algorithm}
\end{algorithm}

\GN{
The algorithmic updates of \algacronym/ are inspired to the scheme proposed
in~\cite{notarnicola2017constraint}, where different agent states are considered
in place of $\by_i$. This notational variation reflects the different analysis approach
of \algacronym/ based on primal decomposition.}

Some appealing features of the \algacronym/ are worth highlighting.
The algorithm naturally preserves privacy of all the agents, in the sense
they do not communicate any of their private information (such as the 
local cost $f_i$, the local constraint $X_i$ or the local solution 
estimate $\bx_i^t$).
In addition, the algorithm is scalable, i.e., the amount of local computation
only depends on the number of neighbors and not on the network size.

In order to state the main result of this paper, let us make the
following assumption on the step-size sequence.
\begin{assumption}
  \label{ass:stepsize}
  The step-size sequence $\{\alpha^t\}_{t \ge 0}$, with each $\alpha^t \ge 0$,
  satisfies $\sum_{t=0}^\infty \alpha^t \!=\! \infty$ and
  $\sum_{t=0}^\infty (\alpha^t)^2 \!< \!\infty$.\oprocend
\end{assumption}
Next we provide the convergence properties of \algacronym/.
Despite its simple form, the analysis is quite involved and requires several
technical tools that will be provided in the forthcoming sections.
\begin{theorem}
\label{thm:convergence}
	Let Assumptions~\ref{ass:problem},~\ref{ass:slater},~\ref{ass:iid_var}
	and~\ref{ass:stepsize} hold.
	Moreover, let $\bmu^\star$ be an optimal Lagrange multiplier of
	problem~\eqref{eq:problem_original} associated to the constraint
	$\sum_{i=1}^N \bg_i(\bx_i) \le \0$ and assume $M > \|\bmu^\star\|_1$.
  Consider a sequence $\{\bx_i^t, \rho_i^t\}_{t \ge 0}, \: i \in \until{N}$
  generated by the \algacronym/ algorithm with allocation
  vectors $\by_i^0$ initialized such that $\sum_{i=1}^N \by_i^0 = \0$.
  Then, almost surely,
  \begin{enumerate}
    \item[(i)] $\sum_{i=1}^N \big( f_i(\bx_i^t) + M \rho_i^t \big) \to f^\star$
      as $t \to \infty$, where $f^\star$ is the optimal cost
      of~\eqref{eq:problem_original};
    
    \item[(ii)] every limit point of $\{(\bx_1^t, \ldots\, \bx_N^t)\}_{t \ge 0}$
      is an optimal (feasible) solution of~\eqref{eq:problem_original}.
    \oprocend
  \end{enumerate}
\end{theorem}
\GN{
In principle, in order to satisfy the assumption $M>\|\bmu^\star\|_1$ in
Theorem~\ref{thm:convergence}, knowledge is needed of the dual optimal
solution $\bmu^\star$. However, this is not necessary in practice, as a
lower bound of $M$ can be efficiently computed when a Slater point is known.
In Section~\ref{sec:discussion_M}, we provide a sufficient condition to select
valid values of $M$ without any knowledge on $\bmu^\star$.

Note also that the algorithm does not employ any averaging mechanism typically
appearing in dual algorithms when the cost functions are not strictly convex.
However, thanks to the primal decomposition approach, we are still able to prove
asymptotic feasibility (other than optimality) of the sequence
$\{(\bx_1^t, \ldots\, \bx_N^t)\}_{t \ge 0}$.
As shown in Section~\ref{sec:simulation_comparison_dual_subg}, the absence of
running averages allows for faster practical convergence, compared to existing methods.
}

\GN{
\begin{remark}[Computational load of \algacronym/]
	As many of duality-based distributed algorithms, \algacronym/ requires the repeated
	solution of local optimization problems and also to compute the
	Lagrange multiplier $\bmu_i^t$ associated to the inequality constraint.
  As a matter of fact, the computation of $\bmu_i^t$ has a minor
  impact on the computational load.
	Indeed, if a solver based on interior-point methods is used,
	it will provide $\bmu_i^t$ as a byproduct of the solution process.
	Alternatively, denoting $(\bx_i^t, \rho_i^t)$ the optimal solution at time $t$,
	a Lagrange multiplier $\bmu_i^t$ can be easily computed as the solution
	of a linear system with positivity constraints
	(cf.~\cite[Proposition 5.1.5]{bertsekas1999nonlinear}), i.e.,
	\begin{align*}
	  \bmu_{i,s} (\bg_{i,s}(\bx_i^t) - \by_{i,s}^t - \rho_i^t) = 0
	  \quad \forall \: s, \hspace{0.3cm} \text{with }
	  	\bmu_i \ge 0.
	  	\eqoprocend
	\end{align*}
\end{remark}
}

\subsection{\GN{Preliminaries on} Relaxation and Primal Decomposition}
\label{sec:relaxation_primal_decomp}

\GN{
In this subsection we recall two preliminary building
blocks for the algorithm analysis, namely the relaxation and 
the primal decomposition approach for problem~\eqref{eq:problem_original}
originally introduced in~\cite{silverman1972primal,bertsekas1999nonlinear,notarnicola2017constraint,camisa2018primal}.}
In a primal decomposition scheme, also called right-hand side allocation,
the coupling constraints $\sum_{i=1}^N \bg_i(\bx_i) \le \0$
are interpreted as a limited resource to be shared among nodes.
A two-level structure is formulated, where independent subproblems,
with a fixed resource allocation, are ``coordinated'' by a master problem
determining the optimal resource allocation.
We will apply such approach to an equivalent, relaxed version of
problem~\eqref{eq:problem_original}.
Formally, consider the following modified version of
problem~\eqref{eq:problem_original},
\begin{align}
\label{eq:problem_relaxed}
\begin{split}
  \min_{\bx_1, \ldots, \bx_N, \rho} \: & \: \sum_{i=1}^N f_i(\bx_i) + M \rho
  \\
  \subj \: & \: \sum_{i=1}^N \bg_i(\bx_i) \le \rho \1,
  \\
  & \: \rho \ge 0, \:\: \bx_i \in X_i, \hspace{0.5cm} i \in \until{N},
\end{split}
\end{align}
where $M > 0$ is a scalar \GN{and we added the
scalar optimization variable $\rho$. In principle, the new variable
allows for a violation of the coupling constraints (in this sense,
we say that problem~\eqref{eq:problem_relaxed} is a relaxed version of
problem~\eqref{eq:problem_original}). However, 
if the constant $M$ appearing in the penalty term $M \rho$ is large enough,
problem~\eqref{eq:problem_relaxed} is equivalent
to~\eqref{eq:problem_original}, as we recall in the next lemma.}
\begin{lemma}[\cite{notarnicola2017constraint}, Proposition III.3]
\label{lemma:relaxation}
  Let Assumptions~\ref{ass:problem} and~\ref{ass:slater} hold.
  Moreover, let $M$ be such that $M > \|\bmu^\star\|_1$, with $\bmu^\star \in \real^S$
  an optimal Lagrange multiplier for problem~\eqref{eq:problem_original}
  associated to the constraint $\sum_{i=1}^N \bg_i(\bx_i) \le \0$.
  Then, the optimal solutions of the relaxed problem~\eqref{eq:problem_relaxed} are
  in the form $(\bx_1^\star, \ldots, \bx_N^\star, 0)$, where $(\bx_1^\star, \ldots, \bx_N^\star)$
  is an optimal solution of~\eqref{eq:problem_original}, i.e., the solutions
  of~\eqref{eq:problem_relaxed} must have $\rho = 0$.
  Moreover, the optimal costs of~\eqref{eq:problem_relaxed}
  and~\eqref{eq:problem_original} are equal.
\oprocend
\end{lemma}

The primal decomposition scheme applied to problem~\eqref{eq:problem_relaxed}
can be formulated as follows.
For all $i \in \until{N}$ and $\by_i \in \real^S$, the $i$-th \emph{subproblem} is
\begin{align}
\label{eq:primal_decomp_subprob}
\begin{split}
  p_i(\by_i) \triangleq \min_{\bx_i, \rho_i} \: & \: f_i(\bx_i) + M \rho_i
  \\
  \subj \: & \: \bg_i(\bx_i) \le \by_i + \rho_i \1
  \\
  & \: \rho_i \ge 0, \:\: \bx_i \in X_i,
\end{split}
\end{align}
where $\by_i \in \real^S$ is a (given) local \emph{allocation} for node $i$ and
$p_i(\by_i)$ denotes the optimal cost as a function of $\by_i$.
The local allocations are ``coordinated'' by the \emph{master} problem, i.e.,
\begin{align}
\label{eq:primal_decomp_master}
\begin{split}
  \min_{\by_1, \ldots, \by_N} \: & \: \sum_{i=1}^N p_i(\by_i)
  \\
  \subj \: & \: \sum_{i=1}^N \by_i = \0.
\end{split}
\end{align}
In the next, we will denote the cost function
of~\eqref{eq:primal_decomp_master} as $p(\by) = \sum_{i=1}^N p_i(\by_i)$,
where $\by = (\by_1, \ldots, \by_N) \in \real^{SN}$.
Notice that subproblem~\eqref{eq:primal_decomp_subprob} is always feasible
for all $\by_i \in \real^S$.
The following lemma establishes the equivalence between the master
problem~\eqref{eq:primal_decomp_master} and the relaxed
problem~\eqref{eq:problem_relaxed}.
\begin{lemma}[\cite{silverman1972primal}]
\label{lemma:primal_decomp_equivalence}
  Let Assumption~\ref{ass:problem} hold. Then,
  problems~\eqref{eq:problem_relaxed} and~\eqref{eq:primal_decomp_master}
  are equivalent, in the sense that \emph{(i)} the optimal costs are equal, \emph{(ii)} if
  $(\bx_1^\star, \ldots, \bx_N^\star)$ is an optimal solution of~\eqref{eq:problem_relaxed} and
  $(\by_1^\star, \ldots, \by_N^\star)$ is an optimal solution of~\eqref{eq:primal_decomp_master},
  then $(\bx_i^\star, 0)$ is an optimal solution of~\eqref{eq:primal_decomp_subprob},
  with $\by_i = \by_i^\star$, for all $i \in \until{N}$.
  \oprocend
\end{lemma}

Thanks to Lemma~\ref{lemma:relaxation} and Lemma~\ref{lemma:primal_decomp_equivalence},
solving problem~\eqref{eq:problem_original} is equivalent to solving
problem~\eqref{eq:primal_decomp_master}. We will show
that indeed the \algacronym/ algorithm solves~\eqref{eq:primal_decomp_master},
thereby indirectly providing a solution to~\eqref{eq:problem_original}.
Consider now the update~\eqref{eq:alg_update}.
Owing to the discussion in~\cite[Section 5.4.4]{bertsekas1999nonlinear},
can be rewritten as
\begin{align*}
  \by_i^{t+1} = \by_i^{t}
    - \alpha^t \sum_{j \in \nbrs_i^t} \big( \subgrad p_i(\by_i^t) - \subgrad p_j(\by_j^t) \big),
\end{align*}
for $i \in \until{N}$. This equivalent form highlights that,
at each iteration $t$, agents adjust their local allocation $\by_i^t$
by performing a subgradient-like step, based only on local and neighboring
information. Note also that, by direct calculation, using the fact that
the underlying graph is undirected, one can see that
$\sum_{i=1}^N \by_i^{t} = \sum_{i=1}^N \by_i^{0} = \0$
for all $t$,
which means that the allocation sequence produced by the algorithm
satisfies the constraint $\sum_{i=1}^N \by_i = \0$ appearing
in problem~\eqref{eq:primal_decomp_master} at each time step $t$.

\GN{
\begin{remark}[On the variables $\rho_i$]
Finally, let us comment on the role of the variables $\rho_i$
appearing in problem~\eqref{eq:alg_local_prob}.
If we impose $\rho_i = 0$, problem~\eqref{eq:alg_local_prob}
may become infeasible for some values of $\by_i$.
Thus, the variable $\rho_i$ guarantees that
the agents can always select a sufficiently large value
of $\rho_i$ in order to satisfy the constraint
$\bg_i(\bx_i) \le \by_i^t + \rho_i\1$.
By Theorem~\ref{thm:convergence}, the sequences $\{\rho_i^t\}_{t \ge 0}$
converge to zero and, hence, they represent only a temporary violation.

Strictly speaking, if one wanted to apply the primal decomposition method
directly to problem~\eqref{eq:problem_original} (or, equivalently,
to problem~\eqref{eq:problem_relaxed} with $\rho_i = 0$),
additional constraints of the type $\by_i \in Y_i$, $i \in \until{N}$
should be included in problem~\eqref{eq:primal_decomp_master},
with each $Y_i$ being the set of $\by_i$ such that the subproblems are
feasible \cite[Section 6.4.2]{bertsekas1999nonlinear}.
However, as it will be clear from the forthcoming analysis, this would prevent us
from obtaining a purely distributed scheme (in particular, problem~\eqref{eq:problem_z}
would not be unconstrained).
\oprocend
\end{remark}
}

\section{Randomized Block Subgradient for Convex Problems}
\label{sec:block_subgrad_method}

In this section, we formulate a (centralized) randomized block subgradient
method for convex problems and formally prove its convergence.
This algorithm will be used in the next to solve an equivalent form of
problem~\eqref{eq:primal_decomp_master}, where the update of blocks is
associated to the activation of edges in the graph.
The results provided here hold for a more general class of optimization
problems, therefore for this section we temporarily stop our
discussion to formalize and analyze the randomized block subgradient method.
Subsequently, we loop back to the main focus of this work and apply the results
of this section for the analysis of \algacronym/.

Let us consider the unconstrained convex problem
\begin{align}
\label{eq:prob_block_subgrad}
\begin{split}
  \min_{\theta \in \real^m} \: & \: \varphi(\theta),
\end{split}
\end{align}
where $\theta$ is the optimization variable and
$\map{\varphi}{\real^m}{\real}$ is a convex function.
We assume that problem~\eqref{eq:prob_block_subgrad} has finite optimal cost,
denoted by $\varphi^\star$, and that (at least) an optimal solution
$\theta^\star \in \real^m$ exists, such that
$\varphi^\star = \varphi(\theta^\star)$.

Let us consider a partition of $\real^m$ into $B \in \natural$ parts, i.e.,
$\real^m = \real^{m_1} \times \cdots \times \real^{m_B}$, such that
$m = \sum_{\ell=1}^B m_\ell$. Therefore, the optimization variable is
the stack of $B$ blocks,
\begin{align*}
  \theta = (\theta_1, \ldots, \theta_B),
\end{align*}
where $\theta_\ell \in \real^{m_\ell}$ for all $\ell \in \until{B}$.
Now, we develop a subgradient method with block-wise updates to solve
problem~\eqref{eq:prob_block_subgrad}. At each iteration $t \in \natural$,
each block $\ell$ is updated with a probability $\prob_\ell > 0$.

We stress that according to the considered model,
blocks can have different update probabilities and
multiple blocks can be updated simultaneously.

For all $t$, we denote by
$B^t \subseteq \until{B}$ the index set of the blocks selected at time $t$.
For all $\ell \in \until{B}$ and $t \ge 0$, let us define
$\rvedge_{\ell}^t$ as the Bernoulli random variable that is equal
to $1$ if $\ell \in B^t$ and $0$ otherwise.
The following assumption is made (compare with Assumption~\ref{ass:iid_var}).
\begin{assumption}
\label{ass:iid_var_blocks}
  For all $\ell \in \until{B}$, the random variables
  $\{\rvedge_{\ell}^t\}_{t \ge 0}$ are independent and identically
  distributed (i.i.d.). Moreover, for all $t \ge 0$, the random variables
  $\{\rvedge_{\ell}^t\}_{\ell \in \until{B}}$ are mutually independent.
  \oprocend
\end{assumption}

The algorithm considered here is based on a subgradient method.
However, at each iteration $t$, only the blocks in $B^t$ are updated, i.e.,
\begin{align}
\theta_\ell^{t+1}
=
\begin{cases}
  \theta_\ell^t - \alpha^t [\subgrad \varphi(\theta^t)]_\ell,
  \hspace{0.5cm}
  & \text{if } \ell \in B^t,
\\
  \theta_\ell^t,
  & \text{if } \ell \notin B^t,
\end{cases}
\label{eq:block_algorithm}
\end{align}
where $\alpha^t$ is the step-size.
Note that algorithm~\eqref{eq:block_algorithm} allows for multiple block updates at once and,
furthermore, blocks have non-uniform update probabilities.
To the best of our knowledge, this general block-subgradient method has not been studied in the literature.
Therefore, we now provide the convergence proof for
algorithm~\eqref{eq:block_algorithm}.
\begin{theorem}
  Let Assumption~\ref{ass:iid_var_blocks} hold and let the step-size sequence
  $\{\alpha^t\}_{t\geq 0}$ satisfy Assumption~\ref{ass:stepsize}.
  Moreover, assume the subgradients of $\varphi$ are block-wise bounded, i.e.,
  assume for all $\ell \in \until{B}$ there exists $C_\ell > 0$ such that
  $\| [ \subgrad \varphi(\theta) ]_\ell \| \le C_\ell$ for all
  $\theta \in \real^{m}$.
  Consider a sequence $\{\theta^t\}_{t \ge 0}$
  generated by algorithm~\eqref{eq:block_algorithm},
  initialized at any $\theta^0 \in \real^{m}$.
  Then, almost surely, it holds
  \begin{align*}
    \lim_{t \to \infty} \varphi(\theta^t) = \varphi^\star.
  \end{align*}
\label{thm:block_subgrad_method}
\end{theorem}
\begin{proof}\renewcommand{\qedsymbol}{}
  To keep the notation light, let us denote the computed subgradients by
	$\beta^t \triangleq \subgrad \varphi(\theta^t)$. Each block $\ell$ is
	denoted by $\beta_\ell^t = [ \subgrad \varphi(\theta^t) ]_\ell$.
	Moreover, for all $\ell \in \until{B}$, let us define the matrix
	$U_{\ell} \in \real^{m \times m}$, obtained by setting to zero
	in the identity matrix all the blocks on the diagonal, except for the
	$\ell$-th block. Thus, when applied to a vector $\theta \in \real^{m}$,
	all the blocks other than the $\ell$-th one are set to zero, i.e.,
	\begin{align*}
	  [ U_{\ell} \theta ]_\kappa
	  =
	  \begin{cases}
	    \theta_\ell & \text{if } \kappa = \ell,
	    \\
	    \0 & \text{otherwise},
	  \end{cases}
	  \hspace{0.5cm}
	  \forall \: \kappa \in \until{B}.
	\end{align*}
	Moreover, for the sake of analysis, let us define
	\begin{align*}
	  W \triangleq \blkdiag \Big( \frac{1}{\prob_1} I_{m_1}, \: \ldots, \: \frac{1}{\prob_B} I_{m_B} \Big),
	\end{align*}
	where $\blkdiag(\cdot)$ is the (block) diagonal operator
	and we recall that $I_{m_\ell}$ is the $m_\ell \times m_\ell$
	identity matrix.
	Note that $W$ is positive definite, thus we can consider the weighted norm
	$\| \theta \|_W$, for which, by definition, it holds
	\begin{align*}
	  \| \theta \|_W^2 = \sum_{\ell=1}^B \frac{\| \theta_\ell \|^2}{\prob_\ell},
	  \hspace{1cm}
	  \theta \in \real^{m}.
	\end{align*}
	
	Next we analyze algorithm~\eqref{eq:block_algorithm}. Let us focus on an
  iteration $t$ and consider any vector $\theta \in \real^{m}$. As
  for the activated blocks $\ell \in B^t$, it holds
	\begin{align*}
	  \|\theta_{\ell}^{t+1} - \theta_{\ell}\|^2
	  &=
	  \|\theta_{\ell}^t - \alpha^t \beta_\ell^t - \theta_{\ell}\|^2
	  \\
	  & =
	  \|\theta_{\ell}^t - \theta_{\ell}\|^2 + (\alpha^t)^2 \|\beta_\ell^t\|^2
	  \\
	  & \hspace{1cm}
	    - 2\alpha^t (\beta_\ell^t)^\top
	      \big( \theta_\ell^t - \theta_\ell \big),
	  \\
	  & \le
	  \|\theta_{\ell}^t - \theta_{\ell}\|^2 + (\alpha^t)^2 C_{\ell}^2
	  \\
	  & \hspace{1cm}
	    - 2\alpha^t U_\ell (\beta^t)^\top
	      \big( \theta^t - \theta \big),
	  \hspace{0.2cm}
	  \forall \: \ell \in B^t,
	\end{align*}
	where $\|\beta_\ell^t\| \leq C_{\ell}$ holds by assumption.
	As for the other blocks $\ell \notin B^t$, we have
	\begin{align*}
	  \|\theta_{\ell}^{t+1} - \theta_{\ell}\|^2
	  &=
	  \|\theta_{\ell}^t - \theta_{\ell}\|^2,
	  \hspace{1cm}
	  \forall \: \ell \notin B^t.
	\end{align*}
	Let us now write the overall evolution in $W$-norm, i.e.,
	\begin{align}
	  \|\theta^{t+1} \!\!- \theta \|_W^2
	  &=
	  \sum_{\ell \in B^t}
	    \frac{\|\theta_{\ell}^{t+1} - \theta_{\ell}\|^2}{\prob_\ell}
	  +
	  \sum_{\ell \notin B^t} 
	    \frac{\|\theta_{\ell}^{t+1} - \theta_{\ell}\|^2}{\prob_\ell}
	  \nonumber
	  \\
	  & \le
	  \sum_{\ell=1}^B \frac{\|\theta_{\ell}^{t} - \theta_{\ell}\|^2}{\prob_\ell}
	  +
	  (\alpha^t)^2 \sum_{\ell \in B^t} \frac{C_{\ell}^2}{\prob_\ell}
	  \nonumber
	  \\
	  & \hspace{0.4cm}
	  - 2 \alpha^t \bigg( \sum_{\ell \in B^t}
	     \frac{1}{\prob_\ell} U_\ell \bigg) (\beta^t)^\top \big( \theta^t - \theta \big)
	  \nonumber
	  \\
	  & \le
	  \|\theta^{t} - \theta \|_W^2
	  +
	  (\alpha^t)^2 C
	  \nonumber
	  \\
	  & \hspace{0.4cm}
	  - 2 \alpha^t \bigg( \sum_{\ell \in B^t}
	    \frac{1}{\prob_\ell} U_\ell \bigg) (\beta^t)^\top \big( \theta^t - \theta \big),
	\label{eq:basic_ineq}
	\end{align}
	where $C \triangleq \sum_{\ell=1}^B \frac{C_\ell^2}{\prob_\ell} > 0$.
	Regarding the matrix $\sum_{\ell \in B^t} \frac{1}{\prob_\ell} U_\ell$
	appearing in~\eqref{eq:basic_ineq}, its expected value is
	\begin{align}
	  \expv\Bigg[ \sum_{\ell \in B^t} \frac{1}{\prob_\ell} U_\ell \Bigg]
	  = \expv\Bigg[ \sum_{\ell=1}^B \frac{\rvedge_\ell^t}{\prob_\ell} U_\ell \Bigg]
	  = \sum_{\ell=1}^B U_\ell = I_{m}.
	\label{eq:expected_value_u_ell}
	\end{align}
	Now, by taking the conditional expectation of~\eqref{eq:basic_ineq} with
  respect to $\FFF^t = \{\theta^0, \ldots, \theta^t\}$ (namely the sequence generated by
  algorithm~\eqref{eq:block_algorithm} up to iteration $t$),
  we obtain for all $\theta \in \real^{m}$ and $t \ge 0$
	\begin{align*}
	  \expv\Big[ \|\theta^{t+1} \!\!- \theta \|_W^2 \: \big\vert \: \FFF^t \Big]
	  &\stackrel{(a)}{\le}
	  \|\theta^{t} - \theta \|_W^2
	    + (\alpha^t)^2 C
	  \\
	  & \hspace{0.6cm}
	    - 2 \alpha^t (\beta^t)^\top \big( \theta^t - \theta \big),
	  \\
	  &\stackrel{(b)}{\le}
    \|\theta^{t} - \theta \|_W^2
	    + (\alpha^t)^2 C
	  \\
	  & \hspace{0.6cm}
	    - 2 \alpha^t \big( \varphi(\theta^t) - \varphi(\theta) \big),
	\end{align*}
	where in $(a)$ we used~\eqref{eq:expected_value_u_ell} and
	the independence of the drawn blocks from the previous iterations
	(cf.~Assumption~\ref{ass:iid_var_blocks}), and $(b)$ follows
	by definition of subgradient of the function $\varphi$.
	By restricting the above inequality to any optimal solution
	$\theta^\star$ of problem~\eqref{eq:problem_z}, we obtain
  \begin{align}
	  \expv\Big[ \|\theta^{t+1} \!\!- \theta^\star \|_W^2 \: \big\vert \: \FFF^t \Big]
	  &\le
    \|\theta^{t} - \theta^\star \|_W^2
	    + (\alpha^t)^2 C
	  \nonumber
	  \\
	  & \hspace{0.9cm}
	    - 2 \alpha^t \big( \varphi(\theta^t) - \varphi^\star \big).
	\label{eq:basic_ineq_exp}
	\end{align}
	Inequality~\eqref{eq:basic_ineq_exp} satisfies the assumptions
	of~\cite[Proposition~8.2.10]{bertsekas2003convex}. Thus,
	by following the same arguments as in~\cite[Proposition 8.2.13]{bertsekas2003convex},
	we conclude that, almost surely,
	\begin{align*}
	  \lim_{t \to \infty} \varphi(\theta^t) = \varphi^\star.
	  \eqoprocend
	\end{align*}
\end{proof}

\begin{remark}
  By employing a different probabilistic model and by slightly adapting the
  previous proof, almost sure cost convergence can also be proved for a block
  subgradient method with single block update, thus complementing,
  e.g., the results in~\cite{dang2015stochastic}.
  \oprocend
\end{remark}

\section{Analysis of \algacronym/}
\label{sec:analysis}

In this section, we provide the analysis of \algacronym/. To this end, we first
reformulate problem~\eqref{eq:primal_decomp_master} by properly exploiting
the graph structure. This reformulation is then used to show that our distributed
algorithm is equivalent to a (centralized) randomized block subgradient method.
We finally rely on the results of Section~\ref{sec:block_subgrad_method} to prove
Theorem~\ref{thm:convergence}.

\subsection{Encoding the Coupling Constraints in Cost Function}
\label{sec:pb_reformulation}
As already mentioned in Section~\ref{sec:relaxation_primal_decomp},
a solution of problem~\eqref{eq:problem_original} can be indirectly obtained
by solving problem~\eqref{eq:primal_decomp_master}.
In order to put problem~\eqref{eq:primal_decomp_master} into a form that is
more convenient for distributed computation, let us apply a graph-induced
change of variables.
Such a manipulation has a twofold benefit: \emph{(i)} it allows for the
suppression and implicit satisfaction of the constraint $\sum_{i=1}^N \by_i = \0$,
\emph{(ii)} it allows for the application of the randomized block subgradient
method to take into account the random activation of edges.

Consider the underlying communication graph $\GG_u$.
Assuming an ordering of the edges, let
$\inc \in \real^{|\EE_u| \times N}$ denote the incidence matrix of $\GG_u$,
where each row (corresponding to an edge in the graph) contains all zero entries
except for the column corresponding to the edge tail (equal to $1$), and for the
column corresponding to the edge head (equal to $-1$).
Namely, if the $k$-th row of $\inc$ corresponds to the edge $(i,j)$, then
the $(k,\ell)$-th entry of $\inc$ is
\begin{align*}
  ( \inc )_{k \ell} =
  \begin{cases}
    1 & \text{if } \ell = i,
    \\
    -1 \hspace{0.2cm} & \text{if } \ell = j,
    \\
    0 & \text{otherwise},
  \end{cases}
\end{align*}
for all $\ell \in \until{N}$.
For all $(i,j) \in \EE_u$, let $\bz_{(ij)} \in \real^{S}$ be a vector associated
to the edge $(i,j)$ and denote by $\bz \in \real^{S |\EE_u|}$ the vector stacking
all $\bz_{(ij)}$, with the same ordering as in $\inc$.
Consider the change of variables for problem~\eqref{eq:primal_decomp_master}
defined through the following linear mapping
\begin{align}
  \by &= \cvmat \bz,
  \hspace{1cm}
  \bz \in \real^{S |\EE_u|},
\label{eq:change_of_variable}
\end{align}
where the matrix $\cvmat$ is defined as
\begin{align}
  \cvmat \triangleq (\inc^\top \kron I_S)  \in \real^{SN \times S |\EE_u|}.
\label{eq:change_variable_matrix}
\end{align}
By using the properties of the Kronecker product, the blocks of $\by$
can be written as
\begin{align*}
  \by_i
  = [ \cvmat \bz ]_i
  = \sum_{j \in \nbrs_{i,u}} \! (\bz_{(ij)} - \bz_{(ji)}),
  \hspace{0.5cm}
  \forall \: i \in \until{N}.
\end{align*}
The next lemma formalizes the fact that the change of
variable~\eqref{eq:change_of_variable} implicitly encodes the
constraint $\sum_{i=1}^N \by_i = \0$.
\begin{lemma}
\label{lemma:properties_change_of_variable}
  The matrix $\cvmat$ in~\eqref{eq:change_variable_matrix} satisfies:
  \begin{itemize}
    \item[(i)] $\sum_{i=1}^N [\cvmat \bz]_i = \0$ for all $\bz \in \real^{S|\EE_u|}$;
    
    \item[(ii)] for all $\tilde{\by} \in \real^{SN}$ satisfying $\sum_{i=1}^N \tilde{\by}_i = \0$
      there exists $\tilde{\bz} \in \real^{S|\EE_u|}$ such that $\tilde{\by} = \cvmat \tilde{\bz}$.
  \end{itemize}
\end{lemma}
\begin{proof}
  To prove \emph{(i)}, we see that
  \begin{align*}
    \sum_{i=1}^N [\cvmat \bz]_i
    &=
    (\1^\top \kron I_S) \cvmat \bz
    \\
    &= (\1^\top  \kron I_S) (\inc^\top \kron I_S) \bz
    \\
    &= \big( (\inc \kron I_S) (\1 \kron I_S) \big)^\top \bz
    \\
    &\stackrel{(a)}{=} \big( (\inc \1) \kron I_S \big)^\top \bz
    \\
    &\stackrel{(b)}{=} \big( \0 \kron I_S \big)^\top \bz = \0,
  \end{align*}
  where in $(a)$ we used the fact $(A \kron B)(C \kron D) = (AC) \kron (BD)$
  since the matrix dimensions are compatible,
  and $(b)$ follows by the property $\inc \1 = \0$ of incidence matrices.
  
  To prove \emph{(ii)}, let $\tilde{\by} \in \real^{SN}$ be such that
  $\sum_{i=1}^N \tilde{\by_i} = \0$, or, equivalently, $(\1^\top \kron I_S) \tilde{\by} = \0$.
  Let us first show that $\bv^\top \tilde{\by} = 0$ for all $\bv \in \Ker (\cvmat^\top)$.
  To this end, take $\bv \in \Ker (\cvmat^\top)$.
  Since $\GG_u$ is connected, then $\rank(\inc) = N-1$.
  Thus, by the properties of the Kronecker product, it holds
  \begin{align*}
    \rank(\cvmat^\top)
    &=
    \rank(\inc \kron I_S)
    \\
    &=
    \rank(\inc) \rank(I_S)
    \\
    &= (N-1)S.
  \end{align*}
  Moreover, by the Rank-Nullity Theorem, it holds
  \begin{align*}
    \dim \Ker(\cvmat^\top)
    &=
    SN - \rank(\cvmat^\top) = S.
  \end{align*}
  But since the columns of $(\1 \kron I_S) \in \real^{SN \times S}$ are linearly
  independent, and since the point \emph{(i)} of the lemma implies that they
  belong to $\Ker(\Pi^\top)$, it follows that they are actually a basis of
  $\Ker(\Pi^\top)$, so that the vector $\bv$ can be written
  as $\bv = (\1 \kron I_S) \blambda$, for some $\blambda \in \real^S$.
  Therefore, it holds
  \begin{align*}
    \bv^\top \tilde{\by}
    = \blambda^\top \underbrace{ (\1^\top \kron I_S) \tilde{\by} }_{= \: \0}
    = 0.
  \end{align*}
  Thus, since $\bv$ is arbitrary, it follows that $\bv^\top \tilde{\by} = 0$
  for all $\bv \in \Ker (\cvmat^\top)$. By definition of orthogonal complement,
  this means that $\tilde{\by} \in \Ker (\cvmat^\top)^{\bot} = \Im(\cvmat)$.
  Equivalently, there exists $\tilde{\bz}$ such that
  $\tilde{\by} = \cvmat \tilde{\bz}$. The proof follows since $\tilde{\by}$
  is arbitrary.
\end{proof}

We now plug the change of variable~\eqref{eq:change_of_variable} into
problem~\eqref{eq:primal_decomp_master}. Formally, for all $i \in \until{N}$,
define the functions
\begin{align*}
  \tp_i \big( \{\bz_{(ij)}, \bz_{(ji)}\}_{j \in \nbrs_{i,u}} \big)
  \triangleq
  p_i\big( [\Pi \bz]_i \big),
  \hspace{0.7cm}
  \bz \in \real^{S|\EE_u|}.
\end{align*}
By Lemma~\ref{lemma:properties_change_of_variable}, we directly
obtain the following result.
\begin{corollary}
\label{corollary:equivalence_master_prob_z}
  Problem~\eqref{eq:primal_decomp_master} is equivalent to the
  unconstrained optimization problem
  \begin{align}
	  \min_{\bz \in \real{S|\EE_u|}} \:
	    \sum_{i=1}^N
	    \tp_i \big( \{\bz_{(ij)}, \bz_{(ji)}\}_{j \in \nbrs_{i,u}} \big),
	\label{eq:problem_z}
	\end{align}
	in the sense that \emph{(i)} the optimal costs are equal, and
  \emph{(ii)} if $\bz^\star$ is an optimal solution
  of~\eqref{eq:problem_z}, then $\by^\star = \cvmat \bz^\star$ is an
  optimal solution of~\eqref{eq:primal_decomp_master}.
	\oprocend
\end{corollary}

In the following, we denote the cost function of~\eqref{eq:problem_z}
as $\tp(\bz) = \sum_{i=1}^N \tp_i \big( \{\bz_{(ij)}, \bz_{(ji)}\}_{j \in \nbrs_{i,u}} \big) = p(\cvmat \bz)$.

\subsection{Equivalence of \algacronym/ and Randomized Block Subgradient}
\label{sec:block_subgrad_method_equiv}

Differently from problem~\eqref{eq:primal_decomp_master},
its equivalent formulation~\eqref{eq:problem_z} is unconstrained.
Hence, it can be solved via subgradient methods without projections steps.
It is possible to exploit the particular structure of problem~\eqref{eq:problem_z}
to recast the random activation of edges as the random update of blocks
within a block subgradient method~\eqref{eq:block_algorithm} applied
to problem~\eqref{eq:problem_z}.
We will use the following identifications,
\begin{align}
  \theta
    = \bz,
  \hspace{0.5cm}
  \text{and}
  \hspace{0.5cm}
  \varphi(\theta)
    = \sum_{i=1}^N \tp_i \big( \{\bz_{(ij)}, \bz_{(ji)}\}_{j \in \nbrs_{i,u}} \big).
\label{eq:identification_z_phi}
\end{align}
As for the block structure, the mapping is as follows.
Each block $\ell \in \until{B}$ of $\bz$, i.e., $\bz_\ell \in \real^{2S}$,
is associated to an undirected edge $(i,j) \in \EE_u$, with $j > i$,
and is defined as
\begin{align}
  \bz_\ell
  =
  \begin{bmatrix}
    \bz_{(ij)}
    \\
    \bz_{(ji)}
  \end{bmatrix}.
\label{eq:identification_blocks}
\end{align}
Therefore, there is a total of $B = |\EE_u|/2$ blocks.
At each iteration $t$, each block $\bz_\ell$ is updated if the corresponding
edge $(i,j) \in \EE^t$, i.e., if $\rvedge_{ij}^t = 1$.
A pictorial representation of the block structure of $\bz$ is provided in
Figure~\ref{fig:blocks}.
\begin{figure}[!htpb]
\centering
  \includegraphics[scale=1]{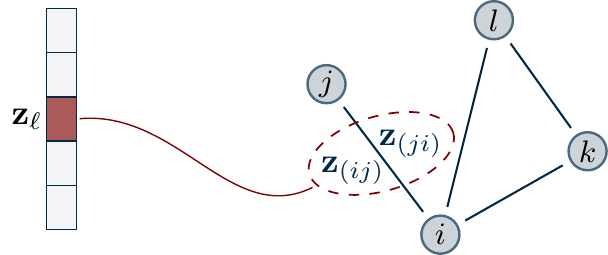}
  \caption{
    Block structure of the variable $\bz$. Each block, say $\ell$, is
    associated to an undirected edge, say $(i,j)$.
    The block is the stack of $\bz_{(ij)}$, associated to the edge $(i,j)$,
    and $\bz_{(ji)}$ associated to the edge $(j,i)$.
  }
\label{fig:blocks}
\end{figure}

\noindent
Consistently with the notation of Section~\ref{sec:block_subgrad_method},
we use the shorthands $\prob_\ell = \prob_{ij}$ and
$\rvedge_\ell^t = \rvedge_{ij}^t$.
At each iteration $t$ of algorithm~\eqref{eq:block_algorithm}, the set
$B^t$ contains all and only the blocks associated
to the edges in $\EE^t$.

Next, we explicitly write the evolution of the sequences
generated by \algacronym/ as a function of the sequences generated by
the block subgradient method~\eqref{eq:block_algorithm}.
For this purpose, let us write a subgradient of $\tp$ at any
$\bz \in \real^{S|\EE_u|}$. By definition, it holds $\tp(\bz) = p(\cvmat \bz)$.
Thus, by using the subgradient property for affine transformations of the
domain\footnote{This property of subgradients is the counterpart of the chain rule
for differentiable functions.}, it holds
\begin{align}
  \subgrad \tp(\bz)
  = (\inc \kron I_S) \subgrad p(\cvmat \bz).
\label{eq:subgrad_p_tilde}
\end{align}
By exploiting the structure of $p$,
the $i$-th block of $\subgrad p(\by)$ is equal to
$\frac{\tilde{\partial} p(\by)}{\partial \by_i} = \subgrad p_i(\by_i)$.
Moreover, since problem~\eqref{eq:primal_decomp_subprob} enjoys strong duality,
a subgradient of $p_i$ at $\by_i$ can be computed as $\subgrad p_i(\by_i) = -\bmu_i$,
where $\bmu_i$ is an optimal Lagrange multiplier of problem~\eqref{eq:primal_decomp_subprob}
(cf.~\cite[Section 5.4.4]{bertsekas1999nonlinear}).
By collecting these facts together with~\eqref{eq:subgrad_p_tilde}, it
follows that the blocks of $\subgrad \tp(\bz)$ can be computed as
\begin{align}
  \frac{\tilde{\partial} \tp(\bz)}{\partial \bz_{(ij)}}
  &=
  \subgrad p_i\Big( [ \cvmat \bz ]_i \Big)
  - \subgrad p_j\Big( [ \cvmat \bz ]_j \Big)
  \nonumber
  \\
  &=
  \bmu_j - \bmu_i,
  \hspace{2cm}
  \forall \: (i,j) \in \EE_u,
\label{eq:subgrad_p_tilde_ij}
\end{align}
where $\frac{\tilde{\partial} \tp(\bz)}{\partial \bz_{(ij)}}$ denotes the
block of $\subgrad \tp(\bz)$ associated to $\bz_{(ij)}$ and,
for all $k \in \until{N}$, $\bmu_k$ denotes an optimal Lagrange
multiplier for the problem
\begin{align}
\label{eq:subgrad_local_prob}
\begin{split}
  \min_{\bx_k, \rho_k} \: & \: f_k(\bx_k) + M \rho_k
  \\
  \subj \: & \: \bg_k(\bx_k) \le [ \cvmat \bz ]_k + \rho_k \1
  \\
  & \: \rho_k \ge 0, \:\: \bx_k \in X_k.
\end{split}
\end{align}
Combining~\eqref{eq:identification_z_phi},~\eqref{eq:identification_blocks}
and~\eqref{eq:subgrad_p_tilde_ij}, the update~\eqref{eq:block_algorithm}
can be recast as
\begin{align}
  \bz_{(ij)}^{t+1}
  =
  \begin{cases}
    \bz_{(ij)}^t + \alpha^t \big( \bmu_i^t - \bmu_j^t \big),
    \hspace{0.5cm}
    & \text{if } (i,j) \in \EE^t,
  \\
    \bz_{(ij)}^t,
    & \text{if } (i,j) \notin \EE^t,
  \end{cases}
\label{eq:alg_z_ij}
\end{align}
where $\bmu_k^t$ denotes an optimal Lagrange multiplier
of~\eqref{eq:subgrad_local_prob} with $\bz = \bz^t$ (with a slight
abuse of notation\footnote{
Indeed, the symbol $\bmu_i^t$ was already defined in
Section~\ref{sec:algorithm_description} in the \algacronym/ table.
In fact, as per the equivalence of the two algorithms (which is being shown here),
the two quantities coincide.}).
Thus,
\begin{align}
  [ \cvmat \bz^{t+1} ]_i
  &=
  \sum_{j \in \nbrs_{i,u}} \! (\bz_{(ij)}^{t+1} - \bz_{(ji)}^{t+1})
  \nonumber
  \\
  &\stackrel{(a)}{=}
  \underbrace{
    \sum_{j \in \nbrs_{i,u}} \! (\bz_{(ij)}^{t} - \bz_{(ji)}^{t})
  }_{\by_i^t}
  +
  2 \alpha^t \sum_{j \in \nbrs_i^t} \big( \bmu_i^t - \bmu_j^t \big)
  \nonumber
  \\
  &= \by_i^{t+1},
  \hspace{2cm}
  i \in \until{N},
  \label{eq:dpd_block_alg_equivalence}
\end{align}
where $(a)$ follows by~\eqref{eq:alg_z_ij}. Therefore the
\algacronym/ algorithm and the block subgradient
method~\eqref{eq:block_algorithm} are equivalent
(up to a factor $2$ in front of the step-size $\alpha^t$, which can
be embedded in its definition).

Before going on, let us state the following technical result.
\begin{lemma}
\label{lemma:bounded_subgradients}
  For all $\bz \in \real^{S|\EE_u|}$, the subgradients of $\tp$ at $\bz$
  are block-wise bounded, i.e.,
  \begin{align*}
    \| [ \subgrad \tp(\bz) ]_\ell \| \le C_\ell,
    \hspace{0.5cm}
    \forall \: \ell \in \until{B}, \forall \: \bz \in \real^{S|\EE|}.
  \end{align*}
  where each $C_\ell > 0$ is a sufficiently large constant 
  proportional to $M$.
\end{lemma}
\begin{proof}
  Fix a block $\ell$ and suppose that it is associated to the edge $(i,j)$.
  According to the previous discussion,
  the $\ell$-th block of $\subgrad \tp(\bz)$ is equal to
  \begin{align*}
    [ \subgrad \tp(\bz) ]_\ell
    =
    \begin{bmatrix}
      \bmu_j - \bmu_i
      \\
      \bmu_i - \bmu_j
    \end{bmatrix},
  \end{align*}
  where each $\bmu_k$ is a Lagrange multiplier of problem~\eqref{eq:subgrad_local_prob}.
  As shown in~\cite[Section III-B]{notarnicola2017constraint}, it holds
  $\|\bmu_k\|_1 \le M$ for all $k \in \until{N}$.
  Thus, the proof follows by using the equivalence of norms and by choosing
  a sufficiently large $C_\ell > 0$.
\end{proof}

\subsection{Proof of Theorem~\ref{thm:convergence}}
\label{sec:proof_theorem_convergence}
The arguments used here rely on the convergence of
the randomized block subgradient method~\eqref{eq:block_algorithm}
and on the algorithm equivalence discussed in
Section~\ref{sec:block_subgrad_method_equiv}.

To prove \emph{(i)}, let us consider the block subgradient method~\eqref{eq:block_algorithm}
applied to problem~\eqref{eq:problem_z}. Note that the function $\tp(\bz)$
is convex (because the functions $p_i$ are convex, cf.
\cite[Section 5.4.4]{bertsekas1999nonlinear}) and its optimal cost is
equal to $f^\star$, the optimal cost of~\eqref{eq:problem_original}
(cf. Corollary~\ref{corollary:equivalence_master_prob_z},
Lemma~\ref{lemma:primal_decomp_equivalence} and Lemma~\ref{lemma:relaxation}).
By Lemma~\ref{lemma:bounded_subgradients} and by the theorem's
assumptions, we can apply Theorem~\ref{thm:block_subgrad_method}
to conclude that, almost surely,
\begin{align*}
  f^\star
  &=
  \lim_{t \to \infty}
    \sum_{i=1}^N \tp_i \big( \{ \bz_{(ij)}^t, \bz_{(ji)}^t \}_{j \in \nbrs_{i,u}} \big)
  \\
  &\stackrel{(a)}{=}
  \lim_{t \to \infty} \sum_{i=1}^N p_i(\by_i^t)
  \\
  &\stackrel{(b)}{=}
  \lim_{t \to \infty} \sum_{i=1}^N \big( f_i(\bx_i^t) + M \rho_i^t \big),
\end{align*}
where $(a)$ follows by definition of $\tp_i$ and
by~\eqref{eq:dpd_block_alg_equivalence} and $(b)$ follows by
construction of $(\bx_i^t, \rho_i^t)$.

To prove \emph{(ii)}, it is possible to follow the same line of proof
of~\cite{notarnicola2017constraint}. However, as here we are considering
a probabilistic setting in a primal decomposition framework, we report
the proof for completeness.
Let us consider the sample set $\bar{\Omega}$ for which point
\emph{(i)} of the theorem holds, and pick any sample path $\omega \in \bar{\Omega}$.
Consider the primal sequence
$\{(\bx_1^t, \ldots, \bx_N^t, \rho_1^t, \ldots, \rho_N^t)\}_{t \ge 0}$ generated
by the \algacronym/ algorithm corresponding to $\omega$.
By summing over $i \in \until{N}$ the inequality
$\bg_i(\bx_i^t) \le \by_i^t + \rho_i^t \1$
(which holds by construction), it holds
\begin{align}
  \sum_{i=1}^N \bg_i(\bx_i^t)
  \le
  \sum_{i=1}^N \by_i^t + \sum_{i=1}^N \rho_i^t \1
  =
  \sum_{i=1}^N \rho_i^t \1.
\label{eq:proof_primal_recovery_t}
\end{align}
Define $\rho^t = \sum_{i=1}^N \rho_i^t$.
By construction, the sequence $\{(\bx_1^t, \ldots, \bx_N^t, \rho^t)\}_{t \ge 0}$ is bounded (as a consequence of point \emph{(i)} and continuity of the functions $f_i(\bx_i)+ M\rho_i$), so that
there exists a sub-sequence of indices $\{t_h\}_{h \ge 0} \subseteq \{t\}_{t \ge 0}$
such that the sequence $\{(\bx_1^{t_n}, \ldots, \bx_N^{t_h}, \rho^{t_h})\}_{h \ge 0}$
converges. Denote the limit point of such sequence as
$(\bar{\bx}_1, \ldots, \bar{\bx}_N, \bar{\rho})$.
From point \emph{(i)} of the theorem, it follows that
\begin{align*} %
  \sum_{i=1}^N f_i(\bar{\bx}_i) + M \bar{\rho} = f^\star.
\end{align*}
By Lemma~\ref{lemma:relaxation}, it must hold $\bar{\rho} = 0$.
As the functions $\bg_i$ are continuous, by taking the limit
in~\eqref{eq:proof_primal_recovery_t} as $h \to \infty$, with $t = t_h$, it holds
\begin{align*} %
  \sum_{i=1}^N \bg_i(\bar{\bx_i}) \le \bar{\rho} \1 = \0.
\end{align*}
Therefore, the point $(\bar{\bx}_1, \ldots, \bar{\bx}_N)$ is an optimal solution
of problem~\eqref{eq:problem_original}.
Since the sample path $\omega \in \bar{\Omega}$ is arbitrary,
every limit point of $\{(\bx_1^t, \ldots, \bx_N^t)\}_{t \ge 0}$
is feasible and cost-optimal for problem~\eqref{eq:problem_original},
almost surely.
\oprocend

\GN{
\subsection{Comparison with Existing Works}
\label{sec:comparison_existing_works}

In this subsection, we are in the position to properly highlight how our
algorithm differs from other works proposed in the literature.
In the special case of static graphs, the algorithm proposed in this
paper can be shown, with an appropriate
change of variables, to have the same evolution of the algorithm
proposed in \cite{notarnicola2017constraint}).
However, several differences are present and are listed hereafter.
First, note that \algacronym/ requires only one communication step
per iteration and $S$ local states, whereas
the algorithm in \cite{notarnicola2017constraint} requires
two communication steps per iteration and has a storage demand
of $2S|\nbrs_i|$ local states.
Moreover, the analysis in \cite{notarnicola2017constraint} relies
on a dual decomposition-based technique which necessarily
freezes the graph topology in the problem formulation and does
not allow for time-varying networks. Instead, in this paper
we consider a primal decomposition approach that allows
us to deal with random, time-varying graphs.

As regards other algorithms working on time-varying networks, 
one can apply a distributed subgradient method to the dual 
of problem~\eqref{eq:problem_original}. In this approach, an averaging 
mechanism to obtain a feasible primal solution is also necessary.
The primal decomposition rationale behind \algacronym/ allows us to
avoid this procedure and obtain a faster convergence rate as shown through 
extensive simulations in Section~\ref{sec:simulation_comparison_dual_subg}.

}

\section{Convergence Rates and Further Discussion}
\label{sec:rates_and_discussion}

In this section, we provide convergence rates of \algacronym/
and a discussion on the parameter $M$.

\subsection{Convergence Rates}
The \algacronym/ algorithm enjoys a sublinear rate for both constant
and diminishing step-size rules. For constant step-size, the cost sequence
converges as $O(1/t)$, while for diminishing step-size, the rate is
$O(1/\log(t))$.
The results provided here are expressed in terms of the quantity
\begin{align*}
  f_\best^t \triangleq \min_{\tau \le t} \sum_{i=1}^N \expv[ f_i(\bx_i^\tau) + M \rho_i^\tau ],
\end{align*}
where the expression in the expected value is the optimal cost of
problem~\eqref{eq:alg_local_prob} for agent $i$ at time $\tau$.
Intuitively, this value represents the best cost value obtained by the algorithm
up to a certain iteration $t$, in an expected sense.

The following analysis is based on deriving convergence rates for our generalized
block subgradient method and thus also complements the ones in, e.g.,
\cite{dang2015stochastic}.
In the next lemma we derive a basic inequality.
\begin{lemma}
  Let Assumptions~\ref{ass:problem},~\ref{ass:slater} and~\ref{ass:iid_var} hold.
  Then, for all $t \ge 0$ it holds
  \begin{align}
	  2 \Bigg( \! \sum_{\tau=0}^t \alpha^\tau \! \Bigg) ( f_\best^t - f^\star )
	  \le 
	  \|\bz^0 - \bz^\star \|_W^2
	    + C \sum_{\tau=0}^t (\alpha^\tau)^2.
	 \label{eq:basic_ineq_convergence_rate}
	\end{align}	 
\end{lemma}
\begin{proof}
  We consider the same line of proof of Theorem~\ref{thm:block_subgrad_method}
  up to~\eqref{eq:basic_ineq_exp}, specialized for $\theta^t = \bz^t$,
  $\theta^\star = \bz^\star$ (an optimal solution of problem~\eqref{eq:problem_z}),
  with corresponding cost $\varphi^\star = \tp(\bz^\star) = f^\star$
  (the optimal cost of problem~\eqref{eq:problem_original}).
	Taking the total expectation (with respect to $\FFF^t$) of~\eqref{eq:basic_ineq_exp},
	it follows that, for all $t \ge 0$,
	\begin{align*}
	  \expv\Big[ \|\bz^{t+1} \!\!- \bz^\star \|_W^2 \Big]
	  &=
	  \expv\Big\{ \expv\Big[ \|\bz^{t+1} \!\!- \bz^\star \|_W^2 \big\vert \FFF^t \Big] \Big\}
	  \\
	  &\le
	    \expv\Big[ \|\bz^{t} - \bz^\star \|_W^2 \Big]
	    + (\alpha^t)^2 C
	  \\
	  & \hspace{0.9cm}
	    - 2 \alpha^t \Big( \expv[ \tp(\bz^t) ] - f^\star \Big).
	\end{align*}
  Applying recursively the previous inequality yields
	\begin{align*}
	  \expv\Big[ \|\bz^{t+1} \!\!- \bz^\star \|_W^2 \Big]
	  &\le
	    \|\bz^0 - \bz^\star \|_W^2
	    + C \sum_{\tau=0}^t (\alpha^\tau)^2
	  \\
	  & \hspace{0.9cm}
	    - 2 \sum_{\tau=0}^t \alpha^\tau \Big( \expv[ \tp(\bz^\tau) ] - f^\star \Big)
	\end{align*}
	for all $t \ge 0$. By using the fact $\|\bz^{t+1} \!\!- \bz^\star \|_W^2 \ge 0$,
	we obtain
	\begin{align*}
	  2 \sum_{\tau=0}^t \alpha^\tau \Big( \expv[ \tp(\bz^\tau) ] - f^\star \Big)
	  &\le
	    \|\bz^0 - \bz^\star \|_W^2
	    + C \sum_{\tau=0}^t (\alpha^\tau)^2,
	\end{align*}
	for all $t \ge 0$. The proof follows by combining the previous inequality with
	$\displaystyle\expv[ \tp(\bz^t) ] \ge \min_{\tau \le t} \expv[ \tp(\bz^\tau) ]$
	and $\tp(\bz^\tau) = p(\by^\tau) = \sum_{i=1}^N p_i(\by_i^\tau)
	= \sum_{i=1}^N f_i(\bx_i^\tau) + M \rho_i^\tau$.
\end{proof}

For constant step-sizes, it is possible to prove a sublinear convergence
rate $O(1/t)$, as formalized next.
\begin{proposition}[Sublinear rate for constant step-size]
  Let the same assumptions of Theorem~\ref{thm:convergence} hold (except
  for Assumption~\ref{ass:stepsize}). Assume $\alpha^t = \alpha > 0$ for all
  $t \ge 0$. Then, it holds
  \begin{align*}
	  f_\best^t - f^\star
	  \le 
	  \frac{\|\bz^0 - \bz^\star \|_W^2}{2 \alpha (t+1)}
	    + \frac{C \alpha}{2}.
  \end{align*}
\label{prop:convergence_rate_constant}
\end{proposition}
\begin{proof}
  It is sufficient to set $\alpha^t = \alpha$
  in~\eqref{eq:basic_ineq_convergence_rate}.
\end{proof}

Note that the previous convergence rate has a term that goes to zero as $t$
goes to infinity, plus a constant (positive) term. In general, without further
assumptions, only convergence within a neighborhood of the optimum can be
proved when a constant step-size is used.

For the case of exact convergence with diminishing step-size, we assume it
has the form $\alpha^t = \frac{K}{t+1}$ with $K > 0$ (which satisfies
Assumption~\ref{ass:stepsize}).
We can obtain a sublinear rate $\OO(1/\log(t))$, as proved next.
\begin{proposition}[Sublinear rate for diminishing step-size]
  Let the same assumptions of Theorem~\ref{thm:convergence} hold.
  Assume $\alpha^t = \frac{K}{t+1}$ for all $t \ge 0$, with $K > 0$.
  Then, it holds
  \begin{align*}
	  f_\best^t - f^\star
	  \le 
	  \frac{ \|\bz^0 - \bz^\star \|_W^2 + C K^2 }
	    { 2 K \log (t+2) }.
  \end{align*}
\label{prop:convergence_rate_diminishing}
\end{proposition}
\begin{proof}
  Let us set $\alpha^t = \frac{K}{t+1}$ in~\eqref{eq:basic_ineq_convergence_rate},
  then it holds
  \begin{align*}
    f_\best^t - f^\star
	  \le 
	  \frac{ \|\bz^0 - \bz^\star \|_W^2 + C K^2 \sum_{\tau=1}^{t+1} \frac{1}{\tau^2} }
	    { 2 K\sum_{\tau=1}^{t+1} \frac{1}{\tau} }.
  \end{align*}
  The proof follows by using the inequalities $\sum_{\tau=1}^{t} \frac{1}{\tau^2} \le 1$
  and $\sum_{\tau=1}^{t} \frac{1}{\tau} \ge \log (t+1)$.
\end{proof}

\begin{remark}
  Convergence rates can be also derived under the assumption of fixed (connected)
  graph by following essentially the same arguments, without block randomization
  in algorithm~\eqref{eq:block_algorithm}. This recovers the approach in
  \cite{notarnicola2017constraint}.
  For constant step-sizes the rate is
  \begin{align*}
    f_\best^t - f^\star
	  \le 
	  \frac{\|\bz^0 - \bz^\star \|^2}{2 \alpha (t+1)}
	    + \frac{C \alpha}{2},
  \end{align*}
  while for diminishing step-sizes the rate is
  \begin{align*}
    f_\best^t - f^\star
	  \le 
	  \frac{ \|\bz^0 - \bz^\star \|^2 + C K^2 }
	    { 2 K \log (t+2) },
  \end{align*}
  where here the quantities $f_\best^t$ and $C$ are defined as
  $f_\best^t \triangleq \min_{\tau \le t} \sum_{i=1}^N f_i(\bx_i^\tau) + M \rho_i^\tau$
  and $C \triangleq \sum_{\ell=1}^B C_\ell^2$.
  \oprocend
\end{remark}

\subsection{Discussion on the Parameter $M$}
\label{sec:discussion_M}

In this subsection, we discuss the choice of the parameter $M$ in the
local minimization problem of the \algacronym/ algorithm (cf.~\eqref{eq:alg_local_prob}).

As per Theorem~\ref{thm:convergence}, it must hold $M > \|\bmu^\star\|_1$,
where $\bmu^\star$ is any dual optimal solution of the original
problem~\eqref{eq:problem_original}. This assumption is needed for
the relaxation approach of Section~\ref{sec:relaxation_primal_decomp} to apply.
In general, a dual optimal solution $\bmu^\star$ of the original
problem~\eqref{eq:problem_original} may not be known in advance. 
However, if a Slater point %
is available (cf.
Assumption~\ref{ass:slater}), it is possible for the agents to compute a
conservative lower bound on $M$.
The next proposition provides a sufficient condition to satisfy $M > \|\bmu^\star\|_1$.
\begin{proposition}
  \label{prop:upper_bound_M}
  Let Assumptions~\ref{ass:problem} and~\ref{ass:slater} hold.
  Moreover, let $(\bar{\bx}_1, \ldots, \bar{\bx}_N)$ be a Slater point, i.e.,
  a feasible point for problem~\eqref{eq:problem_original} with $\sum_{i=1}^N \bg_i(\bar{\bx}_i) < \0$.
  Then, a valid choice of $M$ for Theorem~\ref{thm:convergence} is any
  number satisfying
	\begin{align}
	  M > \frac{1}{\gamma} \sum_{i=1}^N \Big( f_i(\bar{\bx}_i) - \min_{\bx_i \in X_i} f_i(\bx_i) \Big),
	\label{eq:upper_bound_M}
	\end{align}
	where $\gamma = \min_{1 \le s \le S} \{ - \sum_{i=1}^N g_{is} (\bar{\bx}_i) \}$.
\end{proposition}
\begin{proof}
	Let us consider the dual problem associated to~\eqref{eq:problem_original} when
	only the constraint $\sum_{i=1}^N \bg_i(\bx_i) \le \0$ is dualized, i.e.,
	\begin{align}
	\begin{split}
	  \max_{\bmu \in \real^S} \: & \: q(\bmu)
	  \\
	  \subj \: & \: \bmu \ge 0,
	\end{split}
	\label{eq:dual_prob}
	\end{align}
	\GN{with $q(\bmu)$ being the dual function, defined as
	\begin{align*}
	  q(\bmu)
	  &= \inf_{\bx_1 \in X_1, \ldots, \bx_N \in X_N} \bigg\{
	  \sum_{i=1}^N \big( f_i(\bx_i) + \bmu^\top \bg_i(\bx_i) \big) \bigg\}
	  \\
	  &= \sum_{i=1}^N \inf_{\bx_i \in X_i} \big( f_i(\bx_i) + \bmu^\top \bg_i(\bx_i) \big),
	  \\
	  &= \sum_{i=1}^N \min_{\bx_i \in X_i} \big( f_i(\bx_i) + \bmu^\top \bg_i(\bx_i) \big),
	\end{align*}
	where the $\inf$ can be split because the summands
	depend on different variables and the operator $\inf$ can be replaced by $\min$ since
	the sets $X_i$ are compact and $f_i, g_i$
	are continuous due to convexity (cf. Assumption~\ref{ass:problem}).}
	Let us denote by $\bmu^\star$ an optimal solution of problem~\eqref{eq:dual_prob}.
	By Assumptions~\ref{ass:problem} and~\ref{ass:slater}, strong duality holds,
	therefore $q(\bmu^\star) = \sum_{i=1}^N f_i(\bx_i^\star)$, where
	$(\bx_1^\star, \ldots, \bx_N^\star)$ is an optimal solution of
	problem~\eqref{eq:problem_original}. Also, note that $\bmu^\star$ is also
	a Lagrange multiplier of problem~\eqref{eq:problem_original} (see, e.g.,
	\cite[Proposition 5.1.4]{bertsekas1999nonlinear}).
	To upper bound $\|\bmu^\star\|_1$, we invoke
	\cite[Lemma 1]{nedic2009approximate},
	\begin{align}
	  \|\bmu^\star\|_1
	  &\le
	  \frac{1}{\gamma} \Bigg( \sum_{i=1}^N f_i(\bar{\bx}_i) - q(\bmu^\star) \Bigg)
	  \nonumber
	  \\
	  &=
	  \frac{1}{\gamma} \sum_{i=1}^N \big( f_i(\bar{\bx}_i) - f_i(\bx_i^\star) \big)
	  \nonumber
	  \\
	  &\le
	  \frac{1}{\gamma} \sum_{i=1}^N \Big( f_i(\bar{\bx}_i) - \min_{\bx_i \in X_i} f_i(\bx_i) \Big),
	  \label{eq:upper_bound_mustar}
	\end{align}
	\GN{where the minimum in the right-hand side of~\eqref{eq:upper_bound_mustar}
	exists by Weierstrass's Theorem,}
	and the proof follows by choosing $M$ as any number strictly greater
	than the right-hand side of~\eqref{eq:upper_bound_mustar}.
\end{proof}

Note that, if each agent knows its portion $\bar{\bx}_i$ of the Slater vector
$(\bar{\bx}_1, \ldots, \bar{\bx}_N)$, the network can run a combination
of $\min$-consensus and average consensus protocols to determine
the right-hand side of~\eqref{eq:upper_bound_M}, because the quantities in
the sum are locally computable. As such, the calculation of $M$ can be
completely distributed.

\section{Numerical Study}
\label{sec:simulations}

In this section, we show the efficacy of \algacronym/ and validate the theoretical
findings through numerical computations. \GN{We first concentrate on a simple
example to show the main algorithm features. Then, we perform an in-depth
numerical study on an electric vehicle charging scenario.
All the simulations are performed with the \textsc{disropt}
Python package~\cite{farina2019disropt} on a desktop PC, with
MPI-based communication.

\subsection{Basic Example}
We begin by considering a network of $N = 5$ agents that must solve the
convex problem
\begin{align}
\begin{split}
  \min_{\bx_1, \ldots, \bx_N} \: & \: \sum_{i=1}^N \|\bx_i - \br_i\|_1
  \\
  \subj \: & \: \sum_{i=1}^N i \cdot \bx_i \le \0
  \\
  & \: -10 \cdot \1 \le \bx_i \le 10 \cdot \1, \hspace{0.5cm} i \in \until{N},
\end{split}
\label{eq:nonsmooth_example}
\end{align}
where each $\bx_i \in \real^{3}$, and $\br_i \in \real^3$ is a random
vector with entries in the interval $[15, 20]$.
Problem~\eqref{eq:nonsmooth_example} is in the
form~\eqref{eq:problem_original} with the positions
$f_i(\bx_i) = \|\bx_i - \br_i\|_1$,
$X_i = \big\{ \bx_i \in \real^3 | -10 \cdot \1 \le \bx_i \le 10 \cdot \1 \big\}$
and $\bg_i(\bx_i) = i \cdot \bx_i$. 
Note that the objective function and the coupling constraint functions
are convex but not smooth.

As for the communication graph, we generate a random connected graph
with random edge activation probabilities. The resulting edge probability
matrix is
\begin{align*}
  \begin{bmatrix}
    0 & 0 & 0 & 0.5 & 0.6
    \\
    0 & 0 & 0.4 & 0 & 0.7
    \\
    0 & 0.4 & 0 & 0 & 0
    \\
    0.5 & 0 & 0 & 0 & 0
    \\
    0.6 & 0.7 & 0 & 0 & 0
  \end{bmatrix}
\end{align*}

In order to apply the \algacronym/ algorithm, we compute a valid value of
the parameter $M$ appearing in problem~\eqref{eq:alg_local_prob} by using
Proposition~\ref{prop:upper_bound_M}
with the Slater vector $(\bar{\bx}_1, \ldots, \bar{\bx}_N)$ with each
$\bar{\bx}_i = -10 \cdot \1$.
After performing all the computations, we obtain the condition $M > 1$
and we finally choose $M = 6$. The \algacronym/ algorithm
is initialized at $\by_i^0 = 0$ for all $i \in \until{N}$ and the
step-size $\alpha^t = 1/(K+1)^{0.6}$ is used (which satisfies
Assumption~\ref{ass:stepsize}).
The simulation results are reported in Figures~\ref{fig:simulation_nonsmooth_cost}
and~\ref{fig:simulation_nonsmooth_coupling}. The asymptotic behavior of
Theorem~\ref{thm:convergence} is confirmed.

\begin{figure}[!htpb]
\centering
  \includegraphics[scale=1]{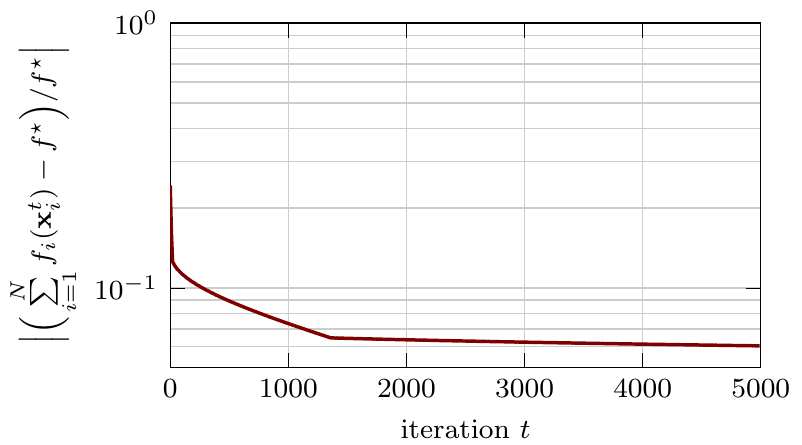}
  \caption{
    \GN{Evolution of the normalized cost error for the basic example.}
  }
\label{fig:simulation_nonsmooth_cost}
\end{figure}

\begin{figure}[!htpb]
\centering
  \includegraphics[scale=1]{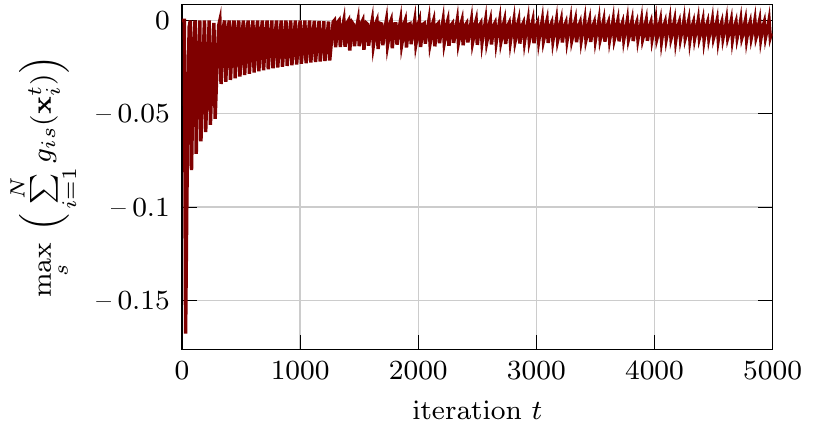}
  \caption{
    \GN{Evolution of the coupling constraint for the basic example.
    A value below zero means that the solution
    computed by the algorithm at that iteration is feasible.}}
\label{fig:simulation_nonsmooth_coupling}
\end{figure}

}

\subsection{Electric Vehicle Charging Problem}
\GN{Let us now} consider the charging of Plug-in Electric Vehicles (PEVs),
which is formulated in detail in~\cite{vujanic2016decomposition} and is slightly
changed here in order to better highlight the algorithm behavior.
\GN{The simulations reported in the remainder of this section are all referred
to this application scenario.}

The problem consists of determining an optimal charging schedule
of $N$ electric vehicles. Each vehicle $i$ has an initial state of charge
$E_i^\text{init}$ and a target state of charge $E_i^\text{ref}$ that must be
reached within a time horizon of $8$ hours, divided into $T = 12$ time slots
of $\Delta T = 40$ minutes. Vehicles must further satisfy a coupling constraint,
which is given by the fact that the total power drawn from the (shared)
electricity grid must not exceed $P^\text{max} = N/2$.
In this paper, we consider the ``charge-only'' case. In order to make
sure the local constraint set are convex (cf. Assumption~\ref{ass:problem}),
we drop the additional integer constraints considered
in~\cite{vujanic2016decomposition}. Thus, the vehicles optimize their charging
rate rather than activating or de-activating the charging mode at each time slot.
Formally, the resulting linear program is
\begin{align*}
\begin{split}
  \min_{\bx_1, \ldots, \bx_N} \: & \: \sum_{i=1}^N c_i^\top \bx_i
  \\
  \subj \: & \: \sum_{i=1}^N A_i \bx_i \le b,
  \\
  & \: \bx_i \in X_i, \hspace{1cm} i \in \until{N},
\end{split}
\end{align*}
where the local constraint sets $X_i$ are compact polyhedra and a total
of $S = 12$ coupling constraints are present. For a complete reference
on the other quantities involved in the problem and not explicitly specified
here, we refer the reader to the extended formulation
in~\cite{vujanic2016decomposition}.

We consider a network of $N = 50$ agents where the underlying graph $\EE_u$
is generated as an Erd\H{o}s-R\'{e}nyi graph with edge probability $0.2$.
The edge activation probabilities $\prob_{ij}$ are randomly picked
in $[0.3, 0.9]$.
In particular, in the next subsections we \emph{(i)} compare our algorithm with
the state of the art, \emph{(ii)} discuss the parameter $M$ and \emph{(iii)}
show the convergence rate.

\subsection{Comparison with State of the Art}
\label{sec:simulation_comparison_dual_subg}

We compare \algacronym/ with the existing Distributed Dual Subgradient
algorithm~\cite{falsone2017dual}.
As for the algorithm tuning \GN{(i.e., the step-size $\alpha^t$
in the update~\eqref{eq:alg_update} and the parameter $M$ appearing
in problem~\eqref{eq:alg_local_prob})}, we choose $M = 30$ and the
diminishing step-size $\alpha^t = \frac{1}{(t+1)^{0.6}}$.
Our algorithm is initialized in $\by_i^0 = \0$ for all $i$ and the Distributed
Dual Subgradient algorithm is initialized in $\blambda_i^0 = \0$ for all $i$.
In Figure~\ref{fig:simulation_cost}, the cost error of both algorithms is shown,
compared with the result of a centralized problem solver. For
our algorithm, the symbol $\bx_i^t$ represents the local solution
of problem~\eqref{eq:alg_local_prob} at time $t$, while for the
Distributed Dual Subgradient, the same symbol represents the (unweighted)
running average of the local solutions over the past iterations.
The figure highlights that, in this simulation, \algacronym/ reached
almost exact cost convergence shortly after $10,000$ iterations with a sudden
change of approximately $10$ orders of magnitude. In principle, for the
Distributed dual subgradient, it is not possible to have such rapid changes
because of the use of running averages.
\begin{figure}[!htpb]
\centering
  \includegraphics[scale=1]{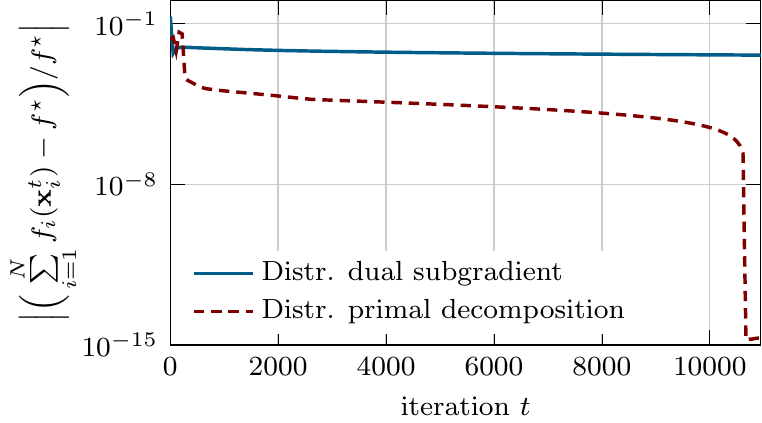}
  \caption{
    Evolution of the normalized cost error for the comparative study
    with the state of the art.
  }
\label{fig:simulation_cost}
\end{figure}

In Figure~\ref{fig:simulation_coupling}, we show the value of the coupling
constraints. The picture highlights that both algorithms are able to provide
feasible solutions within less than 500 iterations, confirming the primal recovery
property.
\begin{figure}[!htpb]
\centering
  \includegraphics[scale=1]{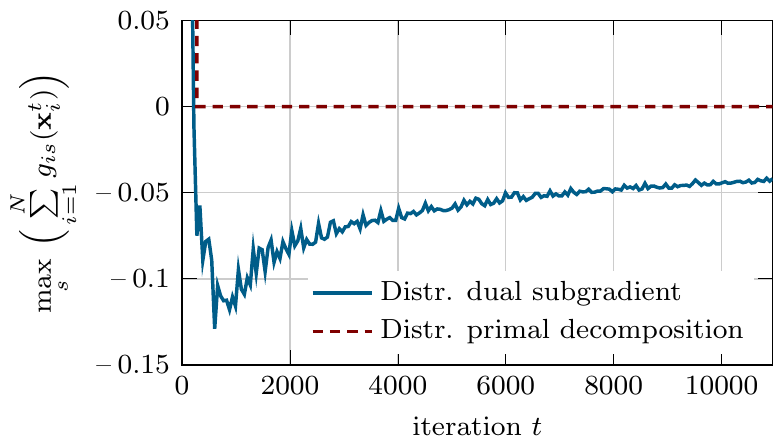}
  \caption{
    Evolution of the coupling constraint for the comparative study
    with the state of the art. A value below zero means that the solution
    computed by the algorithm at that iteration is feasible.}
\label{fig:simulation_coupling}
\end{figure}

\subsection{Impact of the Parameter $M$}
We also perform a numerical comparison of the algorithm behavior for
different values of the parameter $M$ (see also
Section~\ref{sec:discussion_M}).
Under the same set-up of the previous simulation, we \GN{use a
different initialization to guarantee the requirements
imposed by Theorem~\ref{thm:convergence} and also to create some
asymmetry among the initial allocations of the agents. Thus, in this simulation
we consider the initialization rule $\by_i^0 = 5(N - 2i) \1$}
for all $i$, which satisfies $\sum_{i=1}^N \by_i^0 = \0$.

In Figure~\ref{fig:simulation_M_cost} we plot the cost error, including the
extra penalty term $\sum_{i=1}^N M \rho_i^t$, for three different values
of $M$ (all of which satisfy the assumption $M > \|\bmu^\star\|_1$).
It can be seen that the slope of the curve decreases as $M$ increases,
which agrees with the fact that the larger is $M$, the larger is the set in
which subgradients can be found
(Lemma~\ref{lemma:bounded_subgradients}).

\begin{figure}[!htpb]
\centering
  \includegraphics[scale=1]{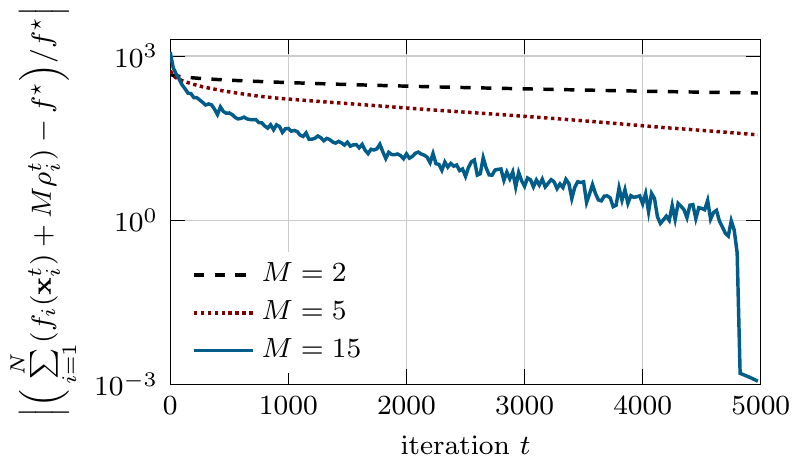}
  \caption{
    Evolution of the normalized cost error for different values of $M$,
    under diminishing step-size.
  }
\label{fig:simulation_M_cost}
\end{figure}

Figure~\ref{fig:simulation_M_rho} shows the maximum value of $\rho_i^t$
among agents. Recall that $\rho_i^t$ is an upper bound on the violation
of the local allocation $\by_i^t$. The picture underlines that such a quantity
is forced to zero faster as $M$ gets bigger. This can be intuitively
explained by the fact that larger values of the penalty $M \rho_i$ drive the
algorithm more quickly towards feasibility of the coupling constraint.

\begin{figure}[!htpb]
\centering
  \includegraphics[scale=1]{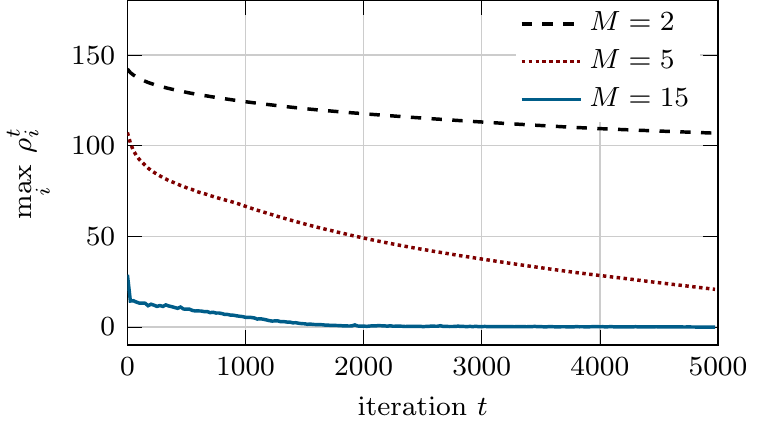}
  \caption{ Evolution of the value of $\max_i \rho_i^t$ for varying values of
    $M$. The quantity represents an upper bound on the coupling constraint
    violation.
  }
\label{fig:simulation_M_rho}
\end{figure}

\subsection{Numerical Study on Convergence Rates}
We finally perform a simulation to point out the different behavior of the
algorithm with constant and diminishing step-sizes.
Under the same set-up of the previous example, with $M = 10$,
we run the algorithm with the diminishing step-size law
$\alpha^t = \frac{0.5}{(t+1)^{0.6}}$ and with the constant
step-size $\alpha^t = 0.01$. As before, agents initialize their local
allocation at $\by_i^0 = 5(N - 2i) \1$ for all $i$.

Figure~\ref{fig:simulation_rate_cost} shows the different algorithm behavior
under the two step-size choices. For constant step-size, the algorithm
converges within a certain tolerance (which is seen in the picture
at around iteration $6,000$), confirming the observations in
Section~\ref{sec:rates_and_discussion}. Moreover, the sublinear behavior
with the diminishing step-size is confirmed.
Interestingly, in this example the constant step-size behaved linearly up to
iteration $4,000$ and superlinearly in the interval $4,000$--$6,000$,
therefore performing much better than the sublinear bound in
Proposition~\ref{prop:convergence_rate_constant}.
\begin{figure}[!htpb]
\centering
  \includegraphics[scale=1]{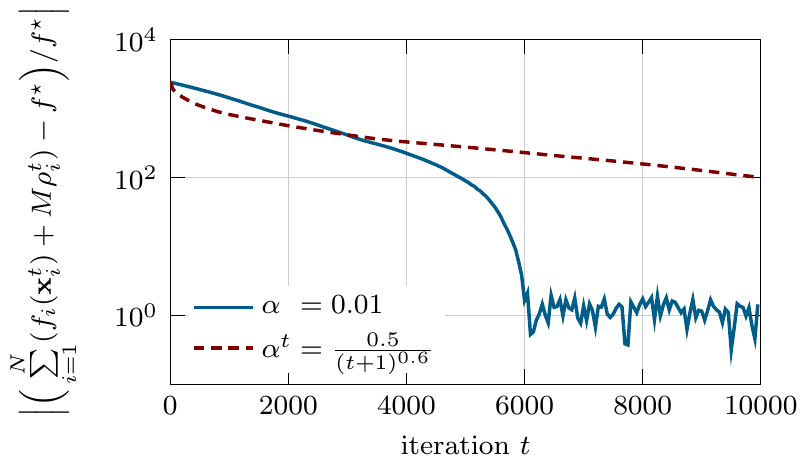}
  \caption{
    Evolution of the cost error for the comparative study on step-sizes.
  }
\label{fig:simulation_rate_cost}
\end{figure}

\section{Conclusions}

In this paper, we presented the \algacronym/ algorithm to solve
constraint-coupled, large-scale, convex optimization problems over random time-varying networks.
The proposed algorithm is based on a relaxation and primal decomposition approach,
and, for the sake of analysis, it is viewed as an instance of a randomized block
subgradient method, in which blocks correspond to edges in the communication graph.
Almost sure convergence to the optimal cost of the original problem
and an almost sure asymptotic primal recovery property are proved.
Sublinear convergence rates are provided under different
step-size assumptions.
Numerical computations on an electric vehicle charging problem
substantiated the theoretical results.

\begin{small}
  \bibliographystyle{IEEEtran}
  \bibliography{primal_decomp_biblio}

\begin{thebibliography}{10}
\providecommand{\url}[1]{#1}
\csname url@samestyle\endcsname
\providecommand{\newblock}{\relax}
\providecommand{\bibinfo}[2]{#2}
\providecommand{\BIBentrySTDinterwordspacing}{\spaceskip=0pt\relax}
\providecommand{\BIBentryALTinterwordstretchfactor}{4}
\providecommand{\BIBentryALTinterwordspacing}{\spaceskip=\fontdimen2\font plus
\BIBentryALTinterwordstretchfactor\fontdimen3\font minus
  \fontdimen4\font\relax}
\providecommand{\BIBforeignlanguage}[2]{{%
\expandafter\ifx\csname l@#1\endcsname\relax
\typeout{** WARNING: IEEEtran.bst: No hyphenation pattern has been}%
\typeout{** loaded for the language `#1'. Using the pattern for}%
\typeout{** the default language instead.}%
\else
\language=\csname l@#1\endcsname
\fi
#2}}
\providecommand{\BIBdecl}{\relax}
\BIBdecl

\bibitem{camisa2019primal}
A.~Camisa, F.~Farina, I.~Notarnicola, and G.~Notarstefano, ``Distributed
  constraint-coupled optimization over random time-varying graphs via primal
  decomposition and block subgradient approaches,'' in \emph{IEEE Conference on
  Decision and Control}, 2019, pp. 6374--6379.

\bibitem{nedic2009distributed}
A.~Nedi{\'c} and A.~Ozdaglar, ``Distributed subgradient methods for multi-agent
  optimization,'' \emph{IEEE Transactions on Automatic Control}, vol.~54,
  no.~1, pp. 48--61, 2009.

\bibitem{duchi2012dual}
J.~C. Duchi, A.~Agarwal, and M.~J. Wainwright, ``Dual averaging for distributed
  optimization: Convergence analysis and network scaling,'' \emph{IEEE
  Transactions on Automatic Control}, vol.~57, no.~3, pp. 592--606, 2012.

\bibitem{zhu2012distributed}
M.~Zhu and S.~Mart{\'i}nez, ``On distributed convex optimization under
  inequality and equality constraints,'' \emph{IEEE Transactions on Automatic
  Control}, vol.~57, no.~1, pp. 151--164, 2012.

\bibitem{mota2013dadmm}
J.~F. Mota, J.~M. Xavier, P.~M. Aguiar, and M.~P{\"u}schel, ``{D}-{ADMM}: A
  communication-efficient distributed algorithm for separable optimization,''
  \emph{IEEE Transactions on Signal Processing}, vol.~61, no.~10, pp.
  2718--2723, 2013.

\bibitem{shi2014linear}
W.~Shi, Q.~Ling, K.~Yuan, G.~Wu, and W.~Yin, ``On the linear convergence of the
  {ADMM} in decentralized consensus optimization,'' \emph{{IEEE} Transactions
  on Signal Processing}, vol.~62, no.~7, pp. 1750--1761, 2014.

\bibitem{jakovetic2014fast}
D.~Jakoveti{\'c}, J.~Xavier, and J.~M. Moura, ``Fast distributed gradient
  methods,'' \emph{IEEE Transactions on Automatic Control}, vol.~59, no.~5, pp.
  1131--1146, 2014.

\bibitem{shi2015extra}
W.~Shi, Q.~Ling, G.~Wu, and W.~Yin, ``{EXTRA}: An exact first-order algorithm
  for decentralized consensus optimization,'' \emph{SIAM Journal on
  Optimization}, vol.~25, no.~2, pp. 944--966, 2015.

\bibitem{simonetto2016primal}
A.~Simonetto and H.~Jamali-Rad, ``Primal recovery from consensus-based dual
  decomposition for distributed convex optimization,'' \emph{Journal of
  Optimization Theory and Applications}, vol. 168, no.~1, pp. 172--197, 2016.

\bibitem{falsone2017dual}
A.~Falsone, K.~Margellos, S.~Garatti, and M.~Prandini, ``Dual decomposition for
  multi-agent distributed optimization with coupling constraints,''
  \emph{Automatica}, vol.~84, pp. 149--158, 2017.

\bibitem{notarnicola2017constraint}
I.~Notarnicola and G.~Notarstefano, ``Constraint-coupled distributed
  optimization: a relaxation and duality approach,'' \emph{{IEEE} Transactions
  on Control of Network Systems}, vol.~PP, no.~99, pp. 1--10, 2019.

\bibitem{chang2014distributed}
T.-H. Chang, A.~Nedi{\'c}, and A.~Scaglione, ``Distributed constrained
  optimization by consensus-based primal-dual perturbation method,'' \emph{IEEE
  Transactions on Automatic Control}, vol.~59, no.~6, pp. 1524--1538, 2014.

\bibitem{mateos2017distributed}
D.~Mateos-N{\'u}nez and J.~Cort{\'e}s, ``Distributed saddle-point subgradient
  algorithms with laplacian averaging,'' \emph{IEEE Transactions on Automatic
  Control}, vol.~62, no.~6, pp. 2720--2735, 2017.

\bibitem{burger2014polyhedral}
M.~B{\"u}rger, G.~Notarstefano, and F.~Allg{\"o}wer, ``A polyhedral
  approximation framework for convex and robust distributed optimization,''
  \emph{IEEE Transactions on Automatic Control}, vol.~59, no.~2, pp. 384--395,
  2014.

\bibitem{liang2019distributed}
S.~Liang, L.~Y. Wang, and G.~Yin, ``Distributed smooth convex optimization with
  coupled constraints,'' \emph{IEEE Transactions on Automatic Control},
  vol.~65, no.~1, pp. 347--353, 2020.

\bibitem{necoara2015linear}
I.~Necoara and V.~Nedelcu, ``On linear convergence of a distributed dual
  gradient algorithm for linearly constrained separable convex problems,''
  \emph{Automatica}, vol.~55, pp. 209--216, 2015.

\bibitem{alghunaim2018dual}
S.~Alghunaim, K.~Yuan, and A.~Sayed, ``Dual coupled diffusion for distributed
  optimization with affine constraints,'' in \emph{IEEE Conference on Decision
  and Control}, 2018, pp. 829--834.

\bibitem{sherson2019distributed}
T.~W. Sherson, R.~Heusdens, and W.~B. Kleijn, ``On the distributed method of
  multipliers for separable convex optimization problems,'' \emph{IEEE
  Transactions on Signal and Information Processing over Networks}, vol.~5,
  no.~3, pp. 495--510, 2019.

\bibitem{chang2014multi}
T.-H. Chang, M.~Hong, and X.~Wang, ``Multi-agent distributed optimization via
  inexact consensus {ADMM},'' \emph{IEEE Transactions on Signal Processing},
  vol.~63, no.~2, pp. 482--497, 2014.

\bibitem{wang2017distributed}
Z.~Wang and C.~J. Ong, ``Distributed model predictive control of linear
  discrete-time systems with local and global constraints,'' \emph{Automatica},
  vol.~81, pp. 184--195, 2017.

\bibitem{carli2019distributed}
R.~Carli and M.~Dotoli, ``Distributed alternating direction method of
  multipliers for linearly constrained optimization over a network,''
  \emph{IEEE Control Systems Letters}, vol.~4, no.~1, pp. 247--252, 2020.

\bibitem{zhang2018consensus}
Y.~Zhang and M.~M. Zavlanos, ``A consensus-based distributed augmented
  lagrangian method,'' in \emph{IEEE Conference on Decision and Control}, 2018,
  pp. 1763--1768.

\bibitem{falsone2019tracking}
A.~Falsone, I.~Notarnicola, G.~Notarstefano, and M.~Prandini, ``Tracking-{ADMM}
  for distributed constraint-coupled optimization,'' \emph{Automatica}, vol.
  117, p. 108962, 2020.

\bibitem{beck2013convergence}
A.~Beck and L.~Tetruashvili, ``On the convergence of block coordinate descent
  type methods,'' \emph{SIAM Journal on Optimization}, vol.~23, no.~4, pp.
  2037--2060, 2013.

\bibitem{razaviyayn2013unified}
M.~Razaviyayn, M.~Hong, and Z.-Q. Luo, ``A unified convergence analysis of
  block successive minimization methods for nonsmooth optimization,''
  \emph{SIAM Journal on Optimization}, vol.~23, no.~2, pp. 1126--1153, 2013.

\bibitem{richtarik2014iteration}
P.~Richt{\'a}rik and M.~Tak{\'a}{\v{c}}, ``Iteration complexity of randomized
  block-coordinate descent methods for minimizing a composite function,''
  \emph{Mathematical Programming}, vol. 144, no. 1-2, pp. 1--38, 2014.

\bibitem{dang2015stochastic}
C.~D. Dang and G.~Lan, ``Stochastic block mirror descent methods for nonsmooth
  and stochastic optimization,'' \emph{SIAM Journal on Optimization}, vol.~25,
  no.~2, pp. 856--881, 2015.

\bibitem{necoara2013random}
I.~Necoara, ``Random coordinate descent algorithms for multi-agent convex
  optimization over networks,'' \emph{IEEE Transactions on Automatic Control},
  vol.~58, no.~8, pp. 2001--2012, 2013.

\bibitem{necoara2014random}
I.~Necoara, Y.~Nesterov, and F.~Glineur, ``A random coordinate descent method
  on large-scale optimization problems with linear constraints,'' Tech. Rep.,
  2014.

\bibitem{bertsekas1999nonlinear}
D.~P. Bertsekas, \emph{Nonlinear programming}.\hskip 1em plus 0.5em minus
  0.4em\relax Athena Scientific, 1999.

\bibitem{silverman1972primal}
G.~J. Silverman, ``Primal decomposition of mathematical programs by resource
  allocation: {I} -- basic theory and a direction-finding procedure,''
  \emph{Operations Research}, vol.~20, no.~1, pp. 58--74, 1972.

\bibitem{camisa2018primal}
A.~Camisa, I.~Notarnicola, and G.~Notarstefano, ``A primal decomposition method
  with suboptimality bounds for distributed mixed-integer linear programming,''
  in \emph{IEEE Conference on Decision and Control}, 2018, pp. 3391--3396.

\bibitem{bertsekas2003convex}
D.~P. Bertsekas, A.~Nedi{\'c}, A.~E. Ozdaglar \emph{et~al.}, \emph{Convex
  analysis and optimization}.\hskip 1em plus 0.5em minus 0.4em\relax Athena
  Scientific, 2003.

\bibitem{nedic2009approximate}
A.~Nedi{\'c} and A.~Ozdaglar, ``Approximate primal solutions and rate analysis
  for dual subgradient methods,'' \emph{SIAM Journal on Optimization}, vol.~19,
  no.~4, pp. 1757--1780, 2009.

\bibitem{farina2019disropt}
F.~Farina, A.~Camisa, A.~Testa, I.~Notarnicola, and G.~Notarstefano,
  ``{DISROPT}: a {P}ython framework for distributed optimization,'' \emph{arXiv
  preprint arXiv:1911.02410}, 2019.

\bibitem{vujanic2016decomposition}
R.~Vujanic, P.~M. Esfahani, P.~J. Goulart, S.~Mari{\'e}thoz, and M.~Morari, ``A
  decomposition method for large scale {MILP}s, with performance guarantees and
  a power system application,'' \emph{Automatica}, vol.~67, pp. 144--156, 2016.

\end{thebibliography}
\end{small}

\end{document}